\documentclass[11pt,reqno]{amsart}
\usepackage[margin=1.1in]{geometry}
\usepackage{amsmath,amssymb,amsthm,mathtools}
\usepackage[activate={true,nocompatibility},final,tracking=false,kerning=false,spacing=false,factor=1100,stretch=10,shrink=10,expansion=false]{microtype}
\usepackage{enumitem}
\usepackage{xcolor}
\usepackage[colorlinks=true,linkcolor=blue!60!black,citecolor=blue!60!black,urlcolor=blue!60!black]{hyperref}

\theoremstyle{plain}
\newtheorem{theorem}{Theorem}[section]
\newtheorem{proposition}[theorem]{Proposition}
\newtheorem{lemma}[theorem]{Lemma}
\newtheorem{corollary}[theorem]{Corollary}
\theoremstyle{definition}
\newtheorem{definition}[theorem]{Definition}
\newtheorem{example}[theorem]{Example}
\newtheorem{remark}[theorem]{Remark}
\newtheorem{convention}[theorem]{Convention}

\newcommand{\g}{\mathfrak{g}}
\newcommand{\gdual}{\mathfrak{g}^{\vee}}
\newcommand{\R}{\mathbb{R}}
\newcommand{\dF}{d_{F}}
\newcommand{\dmu}{d_{\mu}}
\newcommand{\Om}{\Omega}
\newcommand{\Ham}{\mathrm{Ham}}
\newcommand{\Lie}{L_{\infty}}
\newcommand{\vperp}{\varpi}
\newcommand{\io}{\iota}
\newcommand{\Lder}{\mathcal{L}}
\newcommand{\CE}{\mathrm{CE}}
\newcommand{\sgn}{\operatorname{sgn}}
\newcommand{\vs}{\varsigma}
\newcommand{\XF}{\mathfrak{X}(F)}

\newcommand{\Alt}{\operatorname{Alt}}
\newcommand{\id}{\operatorname{id}}

\newcommand{\Ad}{\operatorname{Ad}}
\newcommand{\thL}{\theta^{L}}
\newcommand{\thR}{\theta^{R}}
\newcommand{\ip}[2]{\langle #1,\,#2\rangle}
\newcommand{\CgF}{C_{\g}(F)}
\newcommand{\CGF}{C_{G}(F)}
\newcommand{\dG}{d_{G}}
\newcommand{\dGF}{d_{G}^{F}}
\newcommand{\btot}{\mathbf{d}}

\numberwithin{equation}{section}

\begin{document}

\title[Relative homotopy moment maps]{Relative homotopy moment maps}

\author{Dinamo Djounvouna}
\address{Department of Mathematics, University of Manitoba, Winnipeg, Canada}
\email{djounvod@myumanitoba.ca}


\begin{abstract}
Associated to any smooth map $F\colon M\to N$ equipped with a closed, nondegenerate
relative $(n+1)$-form $\vperp$ -- a \emph{relative $n$-plectic structure} -- is an
$L_\infty$-algebra of relative observables $\Lie(F,\vperp)$, constructed by the author
in earlier work. In this article we develop the corresponding theory of moment maps:
for a Lie group $G$ acting compatibly on $M$ and $N$ and preserving $\vperp$, we define
a \emph{relative homotopy moment map} as an $L_\infty$-morphism from $\g$ into
$\Lie(F,\vperp)$ lifting the infinitesimal action, thereby providing a full relative
generalization of the homotopy moment maps of Callies, Fr\'egier, Rogers and Zambon.
We characterize such morphisms by explicit component equations, show that a relative
homotopy moment map is equivalent to a homotopy moment map on the target $N$ together
with a coherent trivialization of its pullback to $M$, and relate relative moment maps
to a relative Cartan model computing relative equivariant de Rham cohomology.
Every cocycle in the relative Cartan model extending $\vperp$ induces a relative
homotopy moment map via explicit formulas, and we prove the one-step case in full
detail. In the existence theory a new phenomenon appears: under a mild connectivity
hypothesis the Lie-algebra-cohomology obstruction present in the absolute theory
vanishes identically in the relative setting. Finally, we show that quasi-Hamiltonian
$G$-spaces with group-valued moment map $\mu\colon M\to G$ fit into this framework:
the pair $(\eta,\omega)$ built from the Cartan $3$-form is a relative $2$-plectic
structure whose Alekseev--Malkin--Meinrenken axioms amount precisely to a canonical
one-step cocycle in the relative Cartan model, and hence every quasi-Hamiltonian
$G$-space carries a canonical relative homotopy moment map, which we compute
explicitly.
\end{abstract}

\maketitle
\setcounter{tocdepth}{2}
\tableofcontents

\section{Introduction}\label{sec:intro}

A manifold $N$ equipped with a closed $(n+1)$-form $\omega$ which is nondegenerate --
an \emph{$n$-plectic}, or multisymplectic, manifold -- carries a natural higher analogue
of the Poisson algebra of a symplectic manifold: Rogers \cite{Rogers} showed that the
Hamiltonian $(n-1)$-forms on $(N,\omega)$ are the degree-zero component of a Lie
$n$-algebra $\Lie(N,\omega)$, an $L_\infty$-algebra concentrated in degrees
$1-n,\dots,0$ (see also \cite{Zambon}). Building on this, Callies, Fr\'egier, Rogers
and Zambon \cite{CFRZ} introduced \emph{homotopy moment maps}: for a Lie group $G$
acting on $N$ and preserving $\omega$, a homotopy moment map is an $L_\infty$-morphism
\[
  (f)\colon \g \longrightarrow \Lie(N,\omega)
\]
lifting the infinitesimal action $x\mapsto v_x$ through the assignment of Hamiltonian
vector fields. The theory of \cite{CFRZ} -- further developed in
\cite{FLZ,RW,SZ,MS,FRS} -- includes an equivalent description in terms of explicit
component equations, a tight relationship with equivariant de Rham cohomology via the
Cartan model, an obstruction theory for existence, and a large supply of examples.

In parallel, motivated by geometries in which a differential form is closed and
nondegenerate only \emph{relative to a smooth map}, the author introduced in
\cite{DjThesis,DjJGP} the framework of \emph{relative multisymplectic geometry}. Given
a smooth map $F\colon M\to N$, the mapping-cone complex
\[
  \Om^k(F) \;=\; \Om^k(N)\oplus\Om^{k-1}(M),
  \qquad
  \dF(\alpha,\beta)=(d\alpha,\;F^*\alpha-d\beta),
\]
supports a full \emph{relative Cartan calculus} -- relative contraction and Lie
derivative operators for which all the classical Cartan identities, including Cartan's
magic formula, hold unchanged. A closed, nondegenerate relative $(n+1)$-form
$\vperp\in\Om^{n+1}(F)$ is a \emph{relative $n$-plectic structure}, and
\cite{DjJGP} constructs an explicit Lie $n$-algebra of \emph{relative observables}
$\Lie(F,\vperp)$ which recovers Rogers' $\Lie(N,\omega_N)$ in the absolute case where
$M$ is a point. The motivating example is the theory of quasi-Hamiltonian $G$-spaces of
Alekseev, Malkin and Meinrenken \cite{AMM}: for a group-valued moment map
$\mu\colon M\to G$, the pair $\vperp=(\eta,\omega)$ formed by the Cartan $3$-form
$\eta$ on $G$ and the $2$-form $\omega$ on $M$ is a closed, nondegenerate relative
$3$-form, i.e.\ a relative $2$-plectic structure \cite{DjJGP}.

The purpose of the present article is to complete this circle of ideas by developing
the theory of \emph{moment maps} in relative multisymplectic geometry: we provide a
full relative generalization of the homotopy moment maps of \cite{CFRZ}. Throughout,
we work with a Lie group $G$ acting on $M$ and on $N$ in such a way that $F$ is
equivariant, and preserving the relative form $\vperp$.

\subsection{Main results}\label{subsec:mainresults}

Let us describe the contents and main results of the paper, deferring precise
definitions to the body of the text.

\subsubsection*{Relative homotopy moment maps}
A \emph{relative homotopy moment map} (Definition~\ref{def:rhmm}) is an
$L_\infty$-morphism
\[
  (f)\colon \g\longrightarrow \Lie(F,\vperp)
\]
lifting the infinitesimal action of $\g$ by $F$-related pairs of vector fields. In
Proposition~\ref{prop:components} we show, exactly as in the absolute case, that such a
morphism is the same as a collection of maps
$f_k\colon \Lambda^k\g\to\Om^{n-k}(F)$, $1\le k\le n$, satisfying the component
equations
\[
  \sum_{1\le i<j\le k}(-1)^{i+j+1}
  f_{k-1}\bigl([x_i,x_j],x_1,\dots,\widehat{x_i},\dots,\widehat{x_j},\dots,x_k\bigr)
  \;=\;
  \dF f_k(x_1,\dots,x_k)\;+\;\vs(k)\,\io(v_{x_1}\wedge\dots\wedge v_{x_k})\vperp,
\]
for $1\le k\le n+1$ (with $f_0=f_{n+1}:=0$), where
$\vs(k)=-(-1)^{k(k+1)/2}$ and all operators are those of the relative Cartan calculus.
Writing $f_k=(f_k^N,f_k^M)$ in components, we prove a \emph{splitting theorem}
(Theorem~\ref{thm:splitting}): a relative homotopy moment map is precisely a homotopy
moment map $(f^N)$ for the $G$-action on $(N,\omega)$ in the sense of \cite{CFRZ},
together with a coherent, $\eta$-twisted null-homotopy $(f^M)$ of its pullback to $M$.
This makes precise the slogan that relative moment map theory is ``moment map theory on
$N$ trivialized over $M$''.

\subsubsection*{Cohomological framework and the relative Cartan model}
Following \cite{FLZ,CFRZ}, we encode the component equations cohomologically. The
mapping cone of the pullback between the Chevalley--Eilenberg--de Rham bicomplexes of
$N$ and $M$ yields a complex $\CgF$ in which relative moment maps are precisely
primitives of a canonical cocycle determined by $\vperp$
(Proposition~\ref{prop:cohframework}). We then introduce the \emph{relative Cartan
model}
\[
  \CGF[k] \;=\; C^k_G(N)\oplus C^{k-1}_G(M),
  \qquad
  \dGF(\alpha,\beta)=(\dG\alpha,\;F^*\alpha-\dG\beta),
\]
a bicomplex computing relative equivariant de Rham cohomology, and prove that every
$\dGF$-cocycle extending $\vperp$ induces a relative homotopy moment map by explicit
formulas (Theorem~\ref{thm:cartanextension}). For \emph{one-step} extensions
$\vperp-\mathrm{M}$, with $\mathrm{M}\in(\gdual\otimes\Om^{n-1}(F))^G$, the induced
moment map is
\[
  f_k(x_1,\dots,x_k)\;=\;\vs(k)\,
  \io(v_{x_1}\wedge\dots\wedge v_{x_{k-1}})\,\mathrm{M}(x_k),
\]
and we give a complete, self-contained proof of this case
(Theorem~\ref{thm:onestep}), by a direct computation resting on a general
``master identity'' of the relative Cartan calculus (Lemma~\ref{lem:master}).

\subsubsection*{Existence: disappearance of the algebraic obstruction}
In the absolute theory, the existence of a homotopy moment map extending a given
Hamiltonian lift is obstructed both by de Rham cohomology and by a distinguished Lie
algebra cohomology class $[c^{\g}]\in H^{n+2}(\g)$ obtained by evaluating
$\io(v_{x_1}\wedge\dots\wedge v_{x_{n+1}})\omega$ at a point \cite{CFRZ}. In the
relative setting a new structural phenomenon appears
(Theorem~\ref{thm:existence}): if $N$ is connected and $M\neq\emptyset$, then a
$\dF$-closed relative $0$-form necessarily vanishes, and consequently \emph{the
algebraic obstruction disappears}: vanishing of the relative de Rham cohomology
$H^i(\Om(F))$ for $1\le i\le n-1$ already guarantees that every equivariant Hamiltonian
lift extends to a relative homotopy moment map. In the same vein, for relative
pre-$2$-plectic structures and $G$ compact semisimple, a relative homotopy moment map
always exists (Proposition~\ref{prop:compactsemisimple}) -- in contrast with the
absolute case, where \cite{CFRZ} require the existence of zeros of the fundamental
vector fields. We also discuss uniqueness (Proposition~\ref{prop:uniqueness}).

\subsubsection*{Quasi-Hamiltonian $G$-spaces}
Section~\ref{sec:quasiham} contains the flagship application. Let $G$ be a compact
connected Lie group with an $\Ad$-invariant inner product $\ip{\cdot}{\cdot}$, acting
on itself by conjugation, and let $\eta$ denote the Cartan $3$-form. For a
quasi-Hamiltonian $G$-space $(M,\omega,\mu)$ in the sense of \cite{AMM} (axioms stated
in our sign conventions in Definition~\ref{def:qham}), we prove:
\begin{enumerate}[label=(\roman*),leftmargin=2.2em]
\item the pair $\vperp=(\eta,\omega)\in\Om^3(\mu)$ is a relative $2$-plectic structure
      (Theorem~\ref{thm:qham2plectic}), refining \cite{DjJGP};
\item the Alekseev--Malkin--Meinrenken axioms are \emph{equivalent} to the statement
      that
      \[
        \vperp_G \;=\; (\eta,\omega)\;-\;\mathrm{M},
        \qquad
        \mathrm{M}(x)=\Bigl(\tfrac12\ip{\thL+\thR}{x},\,0\Bigr),
      \]
      is a cocycle in the relative Cartan model of degree $3$ extending $\vperp$
      (Theorem~\ref{thm:qhamcocycle});
\item consequently, every quasi-Hamiltonian $G$-space carries a \emph{canonical}
      relative homotopy moment map
      $(f)\colon\g\to\Lie(\mu,\vperp)$ into its Lie $2$-algebra of relative
      observables, with components
      \[
        f_1(x)=\Bigl(\tfrac12\ip{\thL+\thR}{x},\,0\Bigr),
        \qquad
        f_2(x,y)=\Bigl(\tfrac12\ip{(\Ad_g-\Ad_{g^{-1}})x}{y},\,0\Bigr)
      \]
      (Theorem~\ref{thm:qhammm}).
\end{enumerate}
This answers, in the relative framework, the question raised in \cite[\S8.2]{CFRZ}
concerning the precise relationship between quasi-Hamiltonian $G$-spaces and homotopy
moment maps: the group-valued moment map $\mu$ does not merely resemble a moment map --
it \emph{is} the manifestation of a genuine $L_\infty$-morphism, once the
multisymplectic structure is taken relative to $\mu$ itself. We compute the resulting
structure explicitly on conjugacy classes (Example~\ref{ex:conjclass}).

\subsubsection*{Examples and functoriality}
Section~\ref{sec:examples} treats the recovery of the absolute theory, exact relative
structures (with explicit moment maps), and the strict $L_\infty$-projection
$\Lie(F,\vperp)\to\Lie(N,\omega)$; Section~\ref{sec:functoriality} establishes
naturality of the construction under equivariant maps of pairs.

\subsection{Relation to the literature}
This article is a sequel to \cite{DjThesis,DjJGP} and a relative generalization of
\cite{CFRZ}; we have kept the numbering of the guiding absolute statements visible in
the text to facilitate comparison. The cohomological reformulation of moment maps
follows \cite{FLZ}. Multi-moment maps for closed forms in the sense of Madsen--Swann
\cite{MS} and the existence/unicity results of Ryvkin--Wurzbacher \cite{RW} admit
relative counterparts which we briefly indicate. Group-valued moment maps originate in
\cite{AMM}; see \cite{Meinrenken} for a survey. Our conventions for $L_\infty$-algebras
follow \cite{LS,LM,Rogers,CFRZ}.

\subsection{Organization}
Section~\ref{sec:prelim} fixes conventions. Section~\ref{sec:relcartan} develops the
relative Cartan calculus, including the master identity. Section~\ref{sec:relnplectic}
recalls relative $n$-plectic structures and the $L_\infty$-algebra of relative
observables. Section~\ref{sec:rhmm} defines relative homotopy moment maps, derives the
component equations and proves the splitting theorem. Section~\ref{sec:cohomological}
develops the cohomological framework, the relative Cartan model, and the extension
theorems, including the complete proof of the one-step case.
Section~\ref{sec:existence} treats existence, uniqueness and obstructions.
Section~\ref{sec:quasiham} is devoted to quasi-Hamiltonian $G$-spaces.
Sections~\ref{sec:examples}--\ref{sec:functoriality} give further examples and
functoriality.
Appendix~\ref{app:conventions} collects our sign conventions in tabular form.


\section{Preliminaries and conventions}\label{sec:prelim}

All manifolds, maps, forms and actions are smooth; $G$ denotes a Lie group with Lie
algebra $\g$. Our conventions agree with those of \cite{CFRZ}, to which the present
paper reduces in the absolute case; they are collected in
Appendix~\ref{app:conventions}.

\subsection{Graded conventions and the sign $\vs(k)$}\label{subsec:signs}
We work with cohomological $\mathbb{Z}$-grading; $|a|$ denotes the degree of a
homogeneous element. The Koszul sign $\epsilon(\sigma)=\epsilon(\sigma;a_1,\dots,a_k)$
of a permutation $\sigma$ acting on homogeneous elements $a_1,\dots,a_k$ is defined by
$a_{\sigma(1)}\odot\cdots\odot a_{\sigma(k)}
=\epsilon(\sigma)\,a_1\odot\cdots\odot a_k$
in the free graded-commutative algebra; $\chi(\sigma)=\sgn(\sigma)\epsilon(\sigma)$.
An $(i,k-i)$-unshuffle is a permutation $\sigma\in S_k$ with
$\sigma(1)<\dots<\sigma(i)$ and $\sigma(i+1)<\dots<\sigma(k)$.
Throughout the paper we use the sign
\begin{equation}\label{eq:vs}
  \vs(k)\;:=\;-(-1)^{\frac{k(k+1)}{2}},
  \qquad\text{so that}\qquad
  \vs(k-1)\,\vs(k)=(-1)^{k} .
\end{equation}
Thus $\vs(1)=\vs(2)=1$, $\vs(3)=\vs(4)=-1$, $\vs(5)=\vs(6)=1$, and so on. We also set
$\vs(0)=-1$, consistently with \eqref{eq:vs}.

\subsection{$L_\infty$-algebras}\label{subsec:linfty}
We use the conventions of \cite{LS,LM} as in \cite[\S3]{CFRZ}.

\begin{definition}\label{def:linfty}
An \emph{$L_\infty$-algebra} is a graded vector space $L$ equipped with multilinear
maps $l_k\colon L^{\otimes k}\to L$ of degree $2-k$, $k\ge1$, which are graded
skew-symmetric and satisfy, for each $m\ge1$, the generalized Jacobi identity
\begin{equation}\label{eq:genjacobi}
  \sum_{i+j=m+1}\;(-1)^{i(j-1)}
  \sum_{\sigma}\chi(\sigma)\,
  l_j\bigl(l_i(a_{\sigma(1)},\dots,a_{\sigma(i)}),a_{\sigma(i+1)},\dots,a_{\sigma(m)}\bigr)=0,
\end{equation}
where $\sigma$ runs over the $(i,m-i)$-unshuffles. A \emph{Lie $n$-algebra} is an
$L_\infty$-algebra concentrated in degrees $1-n,\dots,0$. An
\emph{$L_\infty$-morphism} $(f)\colon(L,l_k)\rightsquigarrow(L',l'_k)$ is a collection
of graded skew-symmetric maps $f_k\colon L^{\otimes k}\to L'$ of degree $1-k$
satisfying the usual coherence equations \cite[Def.~3.6]{CFRZ}. We will only need the
unpacked form of these equations in the special situation of
Proposition~\ref{prop:components} below, for which we rely on the characterization
given in \cite[Prop.~3.8]{CFRZ}.
\end{definition}

\subsection{Group actions and fundamental vector fields}\label{subsec:actions}

\begin{convention}\label{conv:fundvf}
Let $G$ act on a manifold $P$ (on the left). For $x\in\g$ the \emph{fundamental vector
field} $v_x=v_x^P\in\mathfrak{X}(P)$ is
\begin{equation}\label{eq:fundvf}
  v_x\big|_p\;:=\;\frac{d}{dt}\Big|_{t=0}\exp(-tx)\cdot p ,
\end{equation}
i.e.\ \emph{minus} the infinitesimal generator. With this convention
$x\mapsto v_x$ is a homomorphism of Lie algebras:
$[v_x,v_y]=v_{[x,y]}$ for all $x,y\in\g$. This is the convention of \cite{CFRZ}; it is
opposite to that of \cite{AMM}, a point we return to in Section~\ref{sec:quasiham}.
\end{convention}

If $W$ is a $\g$-module, the Chevalley--Eilenberg complex is
$(\Lambda^\bullet\gdual\otimes W,\delta_{\CE})$; for the trivial module and
$c\in\Lambda^k\gdual$,
\begin{equation}\label{eq:CE}
  (\delta_{\CE}c)(x_1,\dots,x_{k+1})
  =\sum_{1\le i<j\le k+1}(-1)^{i+j}\,
   c\bigl([x_i,x_j],x_1,\dots,\widehat{x_i},\dots,\widehat{x_j},\dots,x_{k+1}\bigr).
\end{equation}

\subsection{Multicontractions}\label{subsec:multicontraction}
For vector fields $w_1,\dots,w_k$ on a manifold $P$ we write
\begin{equation}\label{eq:multicontraction}
  \io(w_1\wedge\dots\wedge w_k)\;:=\;\io_{w_k}\cdots\io_{w_1},
\end{equation}
an operator of degree $-k$ on $\Om(P)$; it is skew in $w_1,\dots,w_k$ since the
$\io_{w_i}$ pairwise anticommute. The same convention will be used for the relative
contraction operators of Section~\ref{sec:relcartan}.

\section{Relative Cartan calculus}\label{sec:relcartan}

This section develops the differential calculus underlying the paper: the mapping-cone
complex of a smooth map and its contraction and Lie derivative operators. The material
extends \cite[\S4.1]{DjJGP}; the general master identity
(Lemma~\ref{lem:master}) in the form given here is what makes the moment map
computations of Sections~\ref{sec:cohomological}--\ref{sec:existence} possible.

\subsection{The mapping-cone complex of a smooth map}\label{subsec:cone}

\begin{definition}\label{def:cone}
Let $F\colon M\to N$ be a smooth map. The \emph{relative de Rham complex} of $F$ is
the mapping cone
\begin{equation}\label{eq:cone}
  \Om^k(F)\;:=\;\Om^k(N)\oplus\Om^{k-1}(M),
  \qquad
  \dF(\alpha,\beta)\;:=\;\bigl(d\alpha,\;F^*\alpha-d\beta\bigr).
\end{equation}
Elements of $\Om^k(F)$ are called \emph{relative $k$-forms}. A relative form with
$\dF(\alpha,\beta)=0$ -- i.e.\ $d\alpha=0$ and $F^*\alpha=d\beta$ -- is
\emph{relatively closed}: $\alpha$ is closed on $N$ and $\beta$ trivializes its
pullback to $M$.
\end{definition}

One checks directly that $\dF^2=0$. The cohomology $H^\bullet(\Om(F),\dF)$ is the
relative de Rham cohomology of $F$; it fits into the long exact sequence
\begin{equation}\label{eq:LES}
  \cdots\longrightarrow H^{k-1}_{\mathrm{dR}}(M)\longrightarrow
  H^{k}(\Om(F))\longrightarrow H^{k}_{\mathrm{dR}}(N)
  \xrightarrow{\;F^*\;} H^{k}_{\mathrm{dR}}(M)\longrightarrow\cdots
\end{equation}
induced by the short exact sequence of complexes
$0\to\Om^{\bullet-1}(M)\to\Om^\bullet(F)\to\Om^\bullet(N)\to0$.
When $F$ is the inclusion of a submanifold this computes the usual relative
cohomology $H^\bullet(N,M)$.

\begin{remark}\label{rem:degreezero}
Note that $\Om^0(F)=C^\infty(N)$ and $\Om^k(F)=0$ for $k<0$. A function
$h\in C^\infty(N)$ is $\dF$-closed if and only if $dh=0$ \emph{and} $F^*h=0$.
Consequently:
\begin{equation}\label{eq:H0vanishing}
  \text{if $N$ is connected and $M\neq\emptyset$, then } H^0(\Om(F))=0 .
\end{equation}
This elementary observation is responsible for the disappearance of the algebraic
obstruction in the relative existence theory; see Theorem~\ref{thm:existence}.
\end{remark}

\subsection{$F$-related pairs, relative contraction and Lie derivative}
\label{subsec:relops}

\begin{definition}\label{def:Fpair}
An \emph{$F$-pair of vector fields} is a pair $v=(v_N,v_M)\in\mathfrak{X}(N)\times
\mathfrak{X}(M)$ such that $v_M$ and $v_N$ are $F$-related:
$TF\circ v_M=v_N\circ F$. We write $\XF$ for the space of $F$-pairs; it is a Lie
algebra under the componentwise bracket
$[u,v]:=([u_N,v_N],[u_M,v_M])$, which is again an $F$-pair.
\end{definition}

\begin{definition}\label{def:reliota}
For an $F$-pair $v=(v_N,v_M)$, the \emph{relative contraction} and \emph{relative Lie
derivative} are the operators on $\Om^\bullet(F)$ of degrees $-1$ and $0$:
\begin{equation}\label{eq:reliotaL}
  \io_v(\alpha,\beta):=\bigl(\io_{v_N}\alpha,\;-\io_{v_M}\beta\bigr),
  \qquad
  \Lder_v(\alpha,\beta):=\bigl(\Lder_{v_N}\alpha,\;\Lder_{v_M}\beta\bigr).
\end{equation}
For $F$-pairs $v_1,\dots,v_k$ we use the multicontraction convention
\eqref{eq:multicontraction}: $\io(v_1\wedge\dots\wedge v_k):=\io_{v_k}\cdots\io_{v_1}$,
so that componentwise
\begin{equation}\label{eq:relmulti}
  \io(v_1\wedge\dots\wedge v_k)(\alpha,\beta)
  =\Bigl(\io\bigl(v_{1,N}\wedge\dots\wedge v_{k,N}\bigr)\alpha,\;
  (-1)^k\,\io\bigl(v_{1,M}\wedge\dots\wedge v_{k,M}\bigr)\beta\Bigr).
\end{equation}
\end{definition}

The sign in the second slot of $\io_v$ is forced by the requirement that Cartan's
magic formula hold on the nose:

\begin{proposition}[Relative Cartan identities]\label{prop:cartanids}
For $F$-pairs $u,v\in\XF$ the following identities hold on $\Om^\bullet(F)$:
\begin{enumerate}[label=\textup{(\roman*)},leftmargin=2.2em]
\item $\dF^2=0$;
\item $\io_u\io_v=-\io_v\io_u$;
\item \textup{(magic formula)}\; $\Lder_v=\dF\,\io_v+\io_v\,\dF$;
\item $[\Lder_u,\io_v]=\io_{[u,v]}$;
\item $[\Lder_u,\Lder_v]=\Lder_{[u,v]}$;
\item $[\dF,\Lder_v]=0$.
\end{enumerate}
\end{proposition}

\begin{proof}
(i) $\dF^2(\alpha,\beta)=(d^2\alpha,\,F^*d\alpha-d(F^*\alpha-d\beta))
=(0,\,F^*d\alpha-dF^*\alpha)=0$ since pullback commutes with $d$.

(ii) is immediate from \eqref{eq:reliotaL}: the sign $(-1)$ in the second slot appears
twice.

(iii) Using $F$-relatedness in the form $\io_{v_M}\,F^*=F^*\,\io_{v_N}$:
\begin{align*}
\dF\io_v(\alpha,\beta)&=\dF(\io_{v_N}\alpha,\,-\io_{v_M}\beta)
=\bigl(d\io_{v_N}\alpha,\;F^*\io_{v_N}\alpha+d\io_{v_M}\beta\bigr),\\
\io_v\dF(\alpha,\beta)&=\io_v(d\alpha,\,F^*\alpha-d\beta)
=\bigl(\io_{v_N}d\alpha,\;-\io_{v_M}F^*\alpha+\io_{v_M}d\beta\bigr).
\end{align*}
Adding and applying the classical magic formula on $N$ and on $M$ gives
$(\Lder_{v_N}\alpha,\Lder_{v_M}\beta)=\Lder_v(\alpha,\beta)$.

(iv) Componentwise: on the first slot this is the classical identity on $N$; on the
second slot, $\Lder_{u_M}(-\io_{v_M}\beta)-(-\io_{v_M})\Lder_{u_M}\beta
=-\io_{[u_M,v_M]}\beta$, which is the second slot of $\io_{[u,v]}$.

(v) is componentwise classical, and (vi) follows from (i) and (iii).
\end{proof}

\begin{remark}
By Proposition~\ref{prop:cartanids}, \emph{every} identity of the classical Cartan
calculus that is a formal consequence of (i)--(vi) holds verbatim in
$(\Om(F),\dF,\io,\Lder)$. This principle will be used repeatedly: proofs in
\cite{Rogers,Zambon,CFRZ} which only invoke the Cartan identities transfer to the
relative setting without change.
\end{remark}

\subsection{The master identity}\label{subsec:master}

The following identity, stated here in full generality (arbitrary $\Psi$, no
invariance or closedness assumptions), is the computational engine of the paper. In
the absolute case, its specialization to closed invariant forms is
\cite[Lem.~9.2]{CFRZ}.

\begin{lemma}[Master identity]\label{lem:master}
Let $v_1,\dots,v_m\in\XF$ \textup{(}$m\ge1$\textup{)} and $\Psi\in\Om^\bullet(F)$.
Then
\begin{equation}\label{eq:master}
\begin{aligned}
  \dF\,\io(v_1\wedge\dots\wedge v_m)\Psi
  \;=\;&(-1)^m\,\io(v_1\wedge\dots\wedge v_m)\,\dF\Psi
  \;+\;\sum_{i=1}^{m}(-1)^{m-i}\,
    \io(v_1\wedge\dots\widehat{v_i}\dots\wedge v_m)\,\Lder_{v_i}\Psi\\
  &+\;(-1)^m\sum_{1\le i<j\le m}(-1)^{i+j}\,
    \io\bigl([v_i,v_j]\wedge v_1\wedge\dots\widehat{v_i}\dots\widehat{v_j}\dots
    \wedge v_m\bigr)\Psi .
\end{aligned}
\end{equation}
In particular, if $\dF\Psi=0$ and $\Lder_{v_i}\Psi=0$ for all $i$, then
\begin{equation}\label{eq:masterclosed}
  \dF\,\io(v_1\wedge\dots\wedge v_m)\Psi
  =(-1)^m\sum_{1\le i<j\le m}(-1)^{i+j}\,
   \io\bigl([v_i,v_j]\wedge v_1\wedge\dots\widehat{v_i}\dots\widehat{v_j}\dots
   \wedge v_m\bigr)\Psi .
\end{equation}
\end{lemma}

\begin{proof}
Induction on $m$. For $m=1$, \eqref{eq:master} reads
$\dF\io_{v_1}\Psi=-\io_{v_1}\dF\Psi+\Lder_{v_1}\Psi$, which is the magic formula
(Proposition~\ref{prop:cartanids}(iii)). Assume \eqref{eq:master} for $m-1$ and write
$I_{m-1}:=\io(v_1\wedge\dots\wedge v_{m-1})$, so that
$\io(v_1\wedge\dots\wedge v_m)=\io_{v_m}I_{m-1}$. By the magic formula,
\begin{equation}\label{eq:masterstep}
  \dF\,\io_{v_m}\bigl(I_{m-1}\Psi\bigr)
  =\Lder_{v_m}I_{m-1}\Psi-\io_{v_m}\,\dF\,I_{m-1}\Psi .
\end{equation}
For the first term, commuting $\Lder_{v_m}$ past each contraction via
Proposition~\ref{prop:cartanids}(iv) gives
\[
  \Lder_{v_m}I_{m-1}\Psi
  =I_{m-1}\Lder_{v_m}\Psi
  +\sum_{j=1}^{m-1}
   \io\bigl(v_1\wedge\dots\wedge v_{j-1}\wedge[v_m,v_j]\wedge v_{j+1}\wedge\dots\wedge
   v_{m-1}\bigr)\Psi ,
\]
and moving $[v_m,v_j]$ to the front costs $(-1)^{j-1}$, so the $j$-th summand equals
$(-1)^{j-1}\io([v_m,v_j]\wedge v_1\wedge\dots\widehat{v_j}\dots\wedge v_{m-1})\Psi$.
For the second term of \eqref{eq:masterstep}, insert the inductive hypothesis for
$\dF I_{m-1}\Psi$ and use that $\io_{v_m}$ applied outermost appends $v_m$ as the last
wedge factor. Collecting terms: the $\dF\Psi$-contribution is
$-\io_{v_m}(-1)^{m-1}I_{m-1}\dF\Psi=(-1)^m\io(v_1\wedge\dots\wedge v_m)\dF\Psi$; the
Lie derivative contributions are $I_{m-1}\Lder_{v_m}\Psi$ (the $i=m$ term) together
with $-\io_{v_m}\sum_{i<m}(-1)^{m-1-i}\io(v_1\wedge\dots\widehat{v_i}\dots\wedge
v_{m-1})\Lder_{v_i}\Psi=\sum_{i<m}(-1)^{m-i}\io(v_1\wedge\dots\widehat{v_i}\dots\wedge
v_m)\Lder_{v_i}\Psi$; and the bracket contributions are the pairs $(j,m)$, with
coefficient $(-1)^{j-1}=(-1)^m(-1)^{j+m}$ matching \eqref{eq:master} after replacing
$[v_m,v_j]=-[v_j,v_m]$, together with the pairs $i<j\le m-1$, with coefficient
$-(-1)^{m-1}(-1)^{i+j}=(-1)^m(-1)^{i+j}$, again matching \eqref{eq:master}. This
completes the induction.
\end{proof}

\begin{corollary}\label{cor:bracketpair}
Let $\vperp\in\Om^{n+1}(F)$ be $\dF$-closed and let $u,v\in\XF$ satisfy
$\Lder_u\vperp=\Lder_v\vperp=0$. Then
$\dF\,\io(u\wedge v)\vperp=-\io_{[u,v]}\vperp$.
\end{corollary}

\begin{proof}
Take $m=2$ in \eqref{eq:masterclosed}.
\end{proof}
\section{Relative $n$-plectic structures and their observables}
\label{sec:relnplectic}

We recall from \cite{DjJGP} the basic objects of relative multisymplectic geometry.

\subsection{Relative $n$-plectic structures}

\begin{definition}\label{def:relnplectic}
Let $F\colon M\to N$ be a smooth map and $n\ge1$. A \emph{relative pre-$n$-plectic
structure} on $F$ is a $\dF$-closed relative $(n+1)$-form
$\vperp=(\omega,\eta)\in\Om^{n+1}(F)$, i.e.
\[
  d\omega=0\ \text{ on }N,
  \qquad
  F^*\omega=d\eta\ \text{ on }M .
\]
It is a \emph{relative $n$-plectic structure} (and $(F,\vperp)$ a \emph{relative
$n$-plectic map}) if $\vperp$ is nondegenerate in the following sense: for every
$m\in M$, the linear map
\begin{equation}\label{eq:nondeg}
  T_mM\longrightarrow
  \Lambda^{n}T^*_{F(m)}N\;\oplus\;\Lambda^{n-1}T^*_{m}M,
  \qquad
  w\longmapsto
  \bigl(\io_{T_mF(w)}\,\omega_{F(m)},\;\io_{w}\,\eta_m\bigr),
\end{equation}
is injective.
\end{definition}

\begin{remark}
Nondegeneracy is a condition at points of $M$ only; along $N\setminus F(M)$ the form
$\omega$ may be arbitrarily degenerate. In the absolute case $M=\{*\}$ (or
$M=\emptyset$) one recovers the pre-$n$-plectic manifolds of
\cite{Rogers,CFRZ}; see Example~\ref{ex:absolute}. The quasi-Hamiltonian case, where
nondegeneracy of $\vperp$ encodes the Alekseev--Malkin--Meinrenken minimal degeneracy
condition, is treated in Section~\ref{sec:quasiham}.
\end{remark}

\subsection{Hamiltonian pairs}

\begin{definition}\label{def:hampair}
Let $(F,\vperp)$ be relative pre-$n$-plectic. A relative $(n-1)$-form
$\sigma=(\varphi,\psi)\in\Om^{n-1}(F)$ is \emph{Hamiltonian} if there exists an
$F$-pair $v_\sigma\in\XF$ with
\begin{equation}\label{eq:hamiltonian}
  \dF\sigma \;=\; -\,\io_{v_\sigma}\vperp .
\end{equation}
We call $v_\sigma$ a \emph{Hamiltonian $F$-pair} of $\sigma$ and write
$\Ham^{n-1}(F,\vperp)\subseteq\Om^{n-1}(F)$ for the space of Hamiltonian relative
$(n-1)$-forms, and $\XF_{\Ham}$ for the space of Hamiltonian $F$-pairs.
\end{definition}

If $\vperp$ is nondegenerate, $v_\sigma$ is uniquely determined by $\sigma$, and the
assignment $\sigma\mapsto v_\sigma$ is linear. Unwinding \eqref{eq:hamiltonian},
$\sigma=(\varphi,\psi)$ is Hamiltonian with pair $v=(v_N,v_M)$ if and only if
\begin{equation}\label{eq:hamiltoniancomp}
  d\varphi=-\io_{v_N}\omega
  \quad\text{on }N,
  \qquad
  d\psi=F^*\varphi+\io_{v_M}\eta
  \quad\text{on }M .
\end{equation}
Thus the $N$-component $\varphi$ is Hamiltonian in the sense of \cite{Rogers,CFRZ},
while $\psi$ is a primitive of $F^*\varphi$ corrected by $\eta$.

\begin{lemma}\label{lem:haminvariance}
If $\sigma$ is Hamiltonian with Hamiltonian $F$-pair $v$, then
$\Lder_v\vperp=0$. Moreover, for Hamiltonian $\sigma_1,\sigma_2$ with pairs
$v_1,v_2$, the relative form
$\io(v_1\wedge v_2)\vperp\in\Om^{n-1}(F)$ is Hamiltonian with pair $[v_1,v_2]$.
\end{lemma}

\begin{proof}
By the magic formula and \eqref{eq:hamiltonian},
$\Lder_v\vperp=\dF\io_v\vperp+\io_v\dF\vperp=\dF(-\dF\sigma)=0$. The second claim is
Corollary~\ref{cor:bracketpair}: $\dF\io(v_1\wedge v_2)\vperp=-\io_{[v_1,v_2]}\vperp$.
\end{proof}

\subsection{The Lie $n$-algebra of relative observables}

\begin{theorem}[{\cite[\S4.2]{DjJGP}}]\label{thm:relobs}
Let $(F,\vperp)$ be relative pre-$n$-plectic. There is a Lie $n$-algebra
$\Lie(F,\vperp)=(L,l_k)$ with underlying graded vector space
\[
  L^{0}=\Ham^{n-1}(F,\vperp),
  \qquad
  L^{-i}=\Om^{n-1-i}(F)\quad(1\le i\le n-1),
\]
unary bracket $l_1=\dF$ on $L^{<0}$ \textup{(}and $l_1=0$ on $L^0$ in the sense that
$l_1(\sigma)= \dF\sigma$ for $\sigma\in L^{-1}$ lands in $L^0$\textup{)}, and higher
brackets, for $k\ge2$,
\begin{equation}\label{eq:brackets}
  l_k(\sigma_1,\dots,\sigma_k)
  =\begin{cases}
   \vs(k)\,\io(v_{\sigma_1}\wedge\dots\wedge v_{\sigma_k})\,\vperp
     & \text{if } |\sigma_1|=\dots=|\sigma_k|=0,\\[2pt]
   0 & \text{otherwise},
  \end{cases}
\end{equation}
where $v_{\sigma_i}$ is a Hamiltonian $F$-pair of $\sigma_i$. If $\vperp$ is
nondegenerate the brackets are independent of choices. In the absolute case
$M=\{*\}$ this recovers Rogers' Lie $n$-algebra $\Lie(N,\omega)$
\cite{Rogers}.
\end{theorem}

\begin{proof}[Sketch of proof]
For $k$ arguments of degree $0$, the right-hand side of \eqref{eq:brackets} lies in
$\Om^{n+1-k}(F)=L^{1-k}$ (for $k\le n+1$), and $l_k$ has degree $2-k$; graded
skew-symmetry follows from Proposition~\ref{prop:cartanids}(ii). The generalized
Jacobi identities \eqref{eq:genjacobi} split into three cases according to the degrees
of the arguments, exactly as in \cite[Thm.~5.2]{Rogers} and \cite[Thm.~6.7]{Zambon};
each case reduces, via Lemma~\ref{lem:haminvariance}, to the closed-invariant master
identity \eqref{eq:masterclosed} applied to $\Psi=\vperp$. Since only the Cartan
identities of Proposition~\ref{prop:cartanids} are used, the absolute proofs carry
over verbatim; complete details in the relative setting are given in
\cite[\S4.2]{DjJGP}.
\end{proof}

\begin{remark}\label{rem:hammorphism}
As in \cite[Thm.~4.7]{CFRZ} there is also a Lie $n$-algebra
$\Ham_\infty(F,\vperp)$ built on pairs $(v,\sigma)\in\XF_{\Ham}\ltimes
\Ham^{n-1}(F,\vperp)$ in degree $0$, with the same negative part, together with strict
$L_\infty$-morphisms
\[
  \Lie(F,\vperp)\;\xrightarrow{\ \simeq\ }\;\Ham_\infty(F,\vperp)
  \;\xrightarrow{\ \pi\ }\;\XF_{\Ham},
  \qquad
  \pi(v,\sigma)=v,
\]
the first of which is an $L_\infty$-quasi-isomorphism when $\vperp$ is nondegenerate.
We will use $\pi$ to phrase the lifting property defining moment maps.
\end{remark}

\section{Relative homotopy moment maps}\label{sec:rhmm}

\subsection{Compatible actions}

\begin{definition}\label{def:compatibleaction}
Let $F\colon M\to N$. An \emph{action of $G$ on $F$} is a pair of actions of $G$ on
$M$ and on $N$ making $F$ equivariant. Then for every $x\in\g$ the fundamental vector
fields form an $F$-pair
\[
  v_x:=\bigl(v_x^N,\,v_x^M\bigr)\in\XF,
\]
and $x\mapsto v_x$ is a Lie algebra homomorphism $\g\to\XF$ by
Convention~\ref{conv:fundvf}. The action \emph{preserves} a relative form $\vperp$ if
$\Lder_{v_x}\vperp=0$ for all $x\in\g$ (which holds in particular when the actions
preserve $\omega$ and $\eta$ separately).
\end{definition}

\begin{definition}\label{def:rhmm}
Let $G$ act on $F$ preserving the relative pre-$n$-plectic structure $\vperp$, with
induced Lie algebra morphism $v\colon\g\to\XF$, $x\mapsto v_x$, and assume each $v_x$
is a Hamiltonian $F$-pair. A \emph{relative homotopy moment map} for the action is an
$L_\infty$-morphism
\[
  (f)\colon\ \g\ \rightsquigarrow\ \Lie(F,\vperp)
\]
\emph{lifting the action}, i.e.\ such that the composition
$\g\rightsquigarrow\Lie(F,\vperp)\xrightarrow{\simeq}\Ham_\infty(F,\vperp)
\xrightarrow{\pi}\XF_{\Ham}$ equals $v$; equivalently,
\begin{equation}\label{eq:lifting}
  \dF f_1(x)\;=\;-\,\io_{v_x}\vperp
  \qquad\text{for all }x\in\g .
\end{equation}
\end{definition}

\subsection{Component equations}

Exactly as in the absolute theory, the $L_\infty$-morphism coherences unpack into a
finite list of equations for the components $f_k$.

\begin{proposition}[Component equations]\label{prop:components}
Let $G$ act on $F$ preserving $\vperp$ as above. A relative homotopy moment map is
equivalent to a collection of linear maps
\[
  f_k\colon\ \Lambda^k\g\ \longrightarrow\ \Om^{n-k}(F),
  \qquad 1\le k\le n,
\]
satisfying, for $1\le k\le n+1$ \textup{(}with the conventions $f_0:=0$,
$f_{n+1}:=0$\textup{)},
\begin{equation}\label{eq:componenteqs}
  \sum_{1\le i<j\le k}(-1)^{i+j+1}
  f_{k-1}\bigl([x_i,x_j],x_1,\dots,\widehat{x_i},\dots,\widehat{x_j},\dots,x_k\bigr)
  \;=\;
  \dF f_k(x_1,\dots,x_k)\;+\;\vs(k)\,
  \io(v_{x_1}\wedge\dots\wedge v_{x_k})\,\vperp
\end{equation}
for all $x_1,\dots,x_k\in\g$. For $k=1$ this is the lifting condition
\eqref{eq:lifting} \textup{(}note $\vs(1)=1$\textup{)}; for $k=n+1$ the term
$\dF f_{n+1}$ is absent and \eqref{eq:componenteqs} is a constraint on $f_n$.
\end{proposition}

\begin{proof}
The graded vector space underlying $\Lie(F,\vperp)$ is concentrated in degrees
$1-n,\dots,0$ and $\g$ is concentrated in degree $0$ with $l_1^{\g}=0$,
$l_2^{\g}=[\cdot,\cdot]$, $l_{\ge3}^{\g}=0$. A degree count shows that the component
$f_k$ takes values in $L^{1-k}=\Om^{n-k}(F)$ and vanishes for $k>n$ unless $k=n+1$,
where the target would be $\Om^{-1}(F)=0$; hence the data is $f_1,\dots,f_n$.
The characterization of $L_\infty$-morphisms out of a Lie algebra into an
$L_\infty$-algebra whose higher brackets vanish on elements of negative degree
(property satisfied by $\Lie(F,\vperp)$ by \eqref{eq:brackets}) given in
\cite[Prop.~3.8]{CFRZ} applies verbatim, since it is a statement of graded linear
algebra independent of the geometric origin of the target. Substituting
$l_1=\dF$ and $l_k=\vs(k)\io(v_{\cdot}\wedge\dots)\vperp$ yields
\eqref{eq:componenteqs}, with the same bookkeeping as in
\cite[Prop.~5.1 \& Def.~5.2]{CFRZ}.
\end{proof}

\begin{remark}\label{rem:multimoment}
For $k=1$, equation \eqref{eq:componenteqs} says each $f_1(x)$ is a Hamiltonian pair
for $v_x$; for $k=2$ it measures the failure of $x\mapsto f_1(x)$ to be a morphism
into the binary bracket, trivialized by $f_2$; and so on. In particular a relative
homotopy moment map with $f_{\ge2}=0$ is a strict morphism; and the ``top'' equation
$k=n+1$ is the relative analogue of the closing condition of \cite{CFRZ}. When
$n=1$ and $M=\{*\}$, \eqref{eq:componenteqs} reduces to the classical comoment map
equations of symplectic geometry.
\end{remark}

\subsection{The splitting theorem: moment maps on $N$ trivialized over $M$}

Write $f_k=(f_k^N,f_k^M)$ with
$f_k^N\colon\Lambda^k\g\to\Om^{n-k}(N)$ and
$f_k^M\colon\Lambda^k\g\to\Om^{n-k-1}(M)$.

\begin{theorem}[Splitting]\label{thm:splitting}
A collection $(f_k)_{1\le k\le n}$ satisfies the relative component equations
\eqref{eq:componenteqs} if and only if:
\begin{enumerate}[label=\textup{(\alph*)},leftmargin=2.2em]
\item $(f_k^N)_{1\le k\le n}$ is a homotopy moment map for the $G$-action on the
pre-$n$-plectic manifold $(N,\omega)$ in the sense of \textup{\cite[Def.~5.2]{CFRZ}};
\item the maps $f_k^M\colon\Lambda^k\g\to\Om^{n-1-k}(M)$ satisfy, for
$1\le k\le n$,
\begin{equation}\label{eq:Mequations}
  d f_k^M(x_1,\dots,x_k)
  \;=\;
  F^*\!f_k^N(x_1,\dots,x_k)
  \;+\;\vs(k-1)\,\io\bigl(v^M_{x_1}\wedge\dots\wedge v^M_{x_k}\bigr)\eta
  \;-\;\!\!\sum_{1\le i<j\le k}\!\!(-1)^{i+j+1}
  f^M_{k-1}\bigl([x_i,x_j],\dots\bigr),
\end{equation}
together with the top constraint
\begin{equation}\label{eq:Mtop}
  \sum_{1\le i<j\le n+1}(-1)^{i+j+1}
  f^M_{n}\bigl([x_i,x_j],x_1,\dots,\widehat{x_i},\dots,\widehat{x_j},\dots,
  x_{n+1}\bigr)
  \;=\;
  \vs(n)\,\io\bigl(v^M_{x_1}\wedge\dots\wedge v^M_{x_{n+1}}\bigr)\eta .
\end{equation}
\end{enumerate}
In other words, a relative homotopy moment map is exactly a homotopy moment map on
$(N,\omega)$ together with an $\eta$-twisted coherent trivialization of its pullback
to $M$.
\end{theorem}

\begin{proof}
Separate each equation \eqref{eq:componenteqs} into its $\Om(N)$- and
$\Om(M)$-components using \eqref{eq:cone} and \eqref{eq:relmulti}. The $N$-component
reads
\[
  \textstyle\sum_{i<j}(-1)^{i+j+1}f^N_{k-1}([x_i,x_j],\dots)
  = d f^N_k(x_1,\dots,x_k)+\vs(k)\,\io(v^N_{x_1}\wedge\dots\wedge v^N_{x_k})\omega,
\]
which is exactly the system of \cite[Prop.~5.1, eqs.~(16)--(17)]{CFRZ} for
$(N,\omega)$; this proves the equivalence of \eqref{eq:componenteqs}$_N$ with (a).
The $M$-component of the right-hand side of \eqref{eq:componenteqs} is
$F^*f^N_k-d f^M_k+\vs(k)(-1)^k\io(v^M_{x_1}\wedge\dots\wedge v^M_{x_k})\eta$; since
$\vs(k)(-1)^k=\vs(k-1)$ by \eqref{eq:vs}, rearranging gives \eqref{eq:Mequations} for
$1\le k\le n$ and, for $k=n+1$ (where $f_{n+1}=0$ and $F^*f^N_{n+1}=0$),
the constraint \eqref{eq:Mtop}.
\end{proof}

\begin{corollary}\label{cor:strictprojection}
The projection $p_N\colon\Om^\bullet(F)\to\Om^\bullet(N)$, $(\alpha,\beta)\mapsto
\alpha$, restricts to a \emph{strict} $L_\infty$-morphism
$\Lie(F,\vperp)\to\Lie(N,\omega)$ intertwining all brackets, and composition with
$p_N$ carries relative homotopy moment maps for $(F,\vperp)$ to homotopy moment maps
for $(N,\omega)$.
\end{corollary}

\begin{proof}
$p_N$ is a cochain map ($p_N\circ\dF=d\circ p_N$) sending a Hamiltonian pair
$(\varphi,\psi)$ with $F$-pair $(v_N,v_M)$ to the Hamiltonian form $\varphi$ with
Hamiltonian vector field $v_N$ (by \eqref{eq:hamiltoniancomp}), and
$p_N\bigl(\io(v_1\wedge\dots\wedge v_k)\vperp\bigr)
=\io(v_{1,N}\wedge\dots\wedge v_{k,N})\omega$ by \eqref{eq:relmulti}. Hence $p_N$
strictly intertwines $l_1$ and all $l_k$, and the composition statement is the
$N$-half of Theorem~\ref{thm:splitting}.
\end{proof}

\begin{remark}\label{rem:conceptual}
Theorem~\ref{thm:splitting} shows that relative moment map theory refines, rather
than replaces, the absolute theory on the target: the datum on $M$ is a homotopy
trivialization, not an independent moment map. This is the moment-map incarnation of
the general principle that relative geometry is the geometry of the mapping cone. In
Section~\ref{sec:quasiham} the trivializing datum on $M$ will be
\emph{invisible} ($f^M=0$): the entire moment map of a quasi-Hamiltonian $G$-space is
carried by the group $G$ itself, through $\mu$.
\end{remark}
\section{Cohomological framework and the relative Cartan model}
\label{sec:cohomological}

\subsection{Moment maps as relative primitives}\label{subsec:cohframework}

Following \cite{FLZ} and \cite[\S6.1]{CFRZ}, we encode the component equations in a
single complex. For a $G$-manifold $P$ let
$C_{\g}^{\,p,q}(P):=\Lambda^p\gdual\otimes\Om^q(P)$, with total degree $p+q$, total
differential
$\btot:=\delta_{\CE}+(-1)^p d$
on $C^{p,q}_{\g}$, and \emph{insertion operators}
\[
  (\io_{\g}\,\alpha)(x_1,\dots,x_{p+1})
  :=\io_{v_{x_{p+1}}}\bigl(\alpha(x_1,\dots,x_p)\bigr)
  \quad\text{skew-symmetrized, i.e. }
  \io_{\g}^{\,k}\alpha
  \text{ has } (\io_{\g}^{\,k}\alpha)(x_1,\dots,x_{p+k})
  \propto\io(v\wedge\dots)\alpha ,
\]
normalized so that for $\alpha\in\Om^q(P)=C^{0,q}_\g(P)$,
\begin{equation}\label{eq:insertion}
  \bigl(\io_{\g}^{\,k}\,\alpha\bigr)(x_1,\dots,x_k)
  \;=\;\io\bigl(v_{x_1}\wedge\dots\wedge v_{x_k}\bigr)\,\alpha .
\end{equation}
Define the \emph{relative Chevalley--Eilenberg--de Rham complex} of an action of $G$
on $F$ as the mapping cone
\begin{equation}\label{eq:CgF}
  \CgF^{\,k}\;:=\;C^{\,k}_{\g}(N)\ \oplus\ C^{\,k-1}_{\g}(M),
  \qquad
  \btot_F(a,b):=\bigl(\btot a,\;F^*a-\btot b\bigr),
\end{equation}
where $F^*$ acts on the $\Om$-factor. The relative insertions
$\io_{\g}^{\,k}$ act on $\CgF$ through the relative contractions of
Definition~\ref{def:reliota}, i.e.\ componentwise with the sign $(-1)^k$ on the
$M$-slot as in \eqref{eq:relmulti}.

Given a collection $(f_k)_{1\le k\le n}$ as in Proposition~\ref{prop:components},
define its \emph{$\vs$-twist}
\[
  f^{\vs}\;:=\;\sum_{k=1}^{n}\vs(k)\,f_k\ \in\ \CgF^{\,n},
\]
using $\Lambda^k\gdual\otimes\Om^{n-k}(F)\subset\CgF^{\,n}$.

\begin{proposition}\label{prop:cohframework}
The collection $(f_k)$ satisfies the component equations \eqref{eq:componenteqs} if
and only if
\begin{equation}\label{eq:primitive}
  \btot_F\, f^{\vs}
  \;=\;\sum_{k=1}^{n+1}(-1)^{k+1}\,\io_{\g}^{\,k}\,\vperp
  \qquad\text{in }\ \CgF^{\,n+1}.
\end{equation}
In particular, relative homotopy moment maps exist if and only if the right-hand side
of \eqref{eq:primitive} -- which is a $\btot_F$-cocycle whenever the action preserves
$\vperp$ -- is exact, and the set of relative homotopy moment maps is a torsor over
the $\btot_F$-cocycles in degree $n$ concentrated in $\Lambda^{\ge1}\gdual$-weights.
\end{proposition}

\begin{proof}
Identical to \cite[Prop.~2.5]{FLZ} and \cite[Prop.~6.1]{CFRZ}: comparing the
$\Lambda^k\gdual$-components of the two sides of \eqref{eq:primitive} and using
$\vs(k-1)\vs(k)=(-1)^k$ to remove the twist reproduces \eqref{eq:componenteqs}
term by term; the computation involves only $\delta_{\CE}$, $\dF$ and the insertions,
and is insensitive to the relative nature of the coefficients. The cocycle property
of the right-hand side follows from the master identity \eqref{eq:masterclosed}
applied to $\vperp$, as in \cite[\S2.2]{FLZ}.
\end{proof}

\subsection{The relative Cartan model}\label{subsec:cartan}

For a $G$-manifold $P$, the Cartan model of equivariant de Rham cohomology is
\[
  C_G^{\,k}(P)
  =\bigoplus_{2p+q=k}\bigl(S^p\gdual\otimes\Om^q(P)\bigr)^{G},
  \qquad
  (\dG\alpha)(x)=d\bigl(\alpha(x)\bigr)-\io_{v_x}\bigl(\alpha(x)\bigr),
\]
with $H_G^\bullet(P)=H^\bullet(C_G(P),\dG)$ for $G$ compact \cite{GS}. Since $F$ is
equivariant and $v_x^M,v_x^N$ are $F$-related, $F^*\colon C_G(N)\to C_G(M)$ is a
cochain map, and we may form the mapping cone:

\begin{definition}\label{def:relcartan}
The \emph{relative Cartan model} of an action of $G$ on $F\colon M\to N$ is the
bicomplex-totalization
\begin{equation}\label{eq:relcartanmodel}
  \CGF[k]\;:=\;C_G^{\,k}(N)\ \oplus\ C_G^{\,k-1}(M),
  \qquad
  \dGF(\alpha,\beta)\;:=\;\bigl(\dG\alpha,\;F^*\alpha-\dG\beta\bigr).
\end{equation}
Its cohomology $H_G^\bullet(F):=H^\bullet(\CGF,\dGF)$ is the \emph{relative
equivariant de Rham cohomology} of the action; it fits into a long exact sequence
analogous to \eqref{eq:LES} relating $H_G^\bullet(N)$ and $H_G^\bullet(M)$, and, for
$G$ compact, computes the equivariant cohomology of the mapping cone
$H_G^\bullet(\mathrm{Cone}(F))$ in the topological sense.
\end{definition}

Unravelling \eqref{eq:relcartanmodel}, an element of $\CGF[k]$ is a pair of
$G$-equivariant polynomial maps $\g\to\Om(N)$, $\g\to\Om(M)$; in bidegree
$(S^p\gdual,\cdot)$ the differential mixes the relative de Rham differential $\dF$
with the relative contraction $\io_{v_x}$ of Definition~\ref{def:reliota}:
\begin{equation}\label{eq:dGFexplicit}
  \bigl(\dGF(\alpha,\beta)\bigr)(x)
  =\dF\bigl(\alpha(x),\beta(x)\bigr)-\io_{v_x}\bigl(\alpha(x),\beta(x)\bigr).
\end{equation}
This is the ``relative Cartan model bicomplex'' announced in the introduction.

\begin{definition}\label{def:extension}
Let $\vperp\in\Om^{n+1}(F)$ be a $G$-invariant relative pre-$n$-plectic structure. An
\emph{(equivariant) extension} of $\vperp$ is a $\dGF$-cocycle
$\vperp_G\in\CGF[n+1]$ whose polynomial degree zero part is $\vperp$; explicitly
\begin{equation}\label{eq:extension}
  \vperp_G=\vperp-\sum_{i\ge1}\mathrm{P}_i,
  \qquad
  \mathrm{P}_i\in\bigl(S^i\gdual\otimes\Om^{n+1-2i}(F)\bigr)^{G},
  \qquad
  \dGF\vperp_G=0 .
\end{equation}
An extension is \emph{one-step} if $\mathrm{P}_i=0$ for $i\ge2$.
\end{definition}

\subsection{One-step extensions: the fundamental construction}
\label{subsec:onestep}

Polarizing the polynomial variable, a one-step extension is the same as a linear map
$\mathrm{M}\colon\g\to\Om^{n-1}(F)$ subject to three conditions:

\begin{lemma}\label{lem:onestepconditions}
$\vperp-\mathrm{M}$ is a one-step extension of $\vperp$ if and only if for all
$x,y\in\g$:
\begin{enumerate}[label=\textup{(M\arabic*)},leftmargin=2.6em]
\item\label{M1} \textup{(Hamiltonicity)}\quad
$\dF\,\mathrm{M}(x)=-\,\io_{v_x}\vperp$;
\item\label{M2} \textup{(equivariance)}\quad
$\Lder_{v_x}\mathrm{M}(y)=\mathrm{M}([x,y])$;
\item\label{M3} \textup{(isotropy)}\quad
$\io_{v_x}\mathrm{M}(x)=0$.
\end{enumerate}
\end{lemma}

\begin{proof}
Expand $\dGF(\vperp-\mathrm{M})(x)$ by \eqref{eq:dGFexplicit} and sort by polynomial
degree in $x$: degree $0$ gives $\dF\vperp=0$ (assumed); degree $1$ gives
$-\dF\mathrm{M}(x)-\io_{v_x}\vperp=0$, i.e.\ \ref{M1}; degree $2$ gives
$\io_{v_x}\mathrm{M}(x)=0$, i.e.\ \ref{M3}. Condition \ref{M2} is the
$G$-invariance (in infinitesimal form) of $\mathrm{M}$ as an element of
$(\gdual\otimes\Om^{n-1}(F))^G$.
\end{proof}

Note that \ref{M1} says precisely that each $\mathrm{M}(x)$ is a Hamiltonian pair for
$v_x$ (Definition~\ref{def:hampair}), while \ref{M3} polarizes, using \ref{M1}-linearity,
to
\begin{equation}\label{eq:polarized}
  \io_{v_x}\mathrm{M}(y)\;=\;-\,\io_{v_y}\mathrm{M}(x)
  \qquad\text{for all }x,y\in\g .
\end{equation}

The next theorem is the relative generalization of \cite[Cor.~6.9]{CFRZ}; we give a
complete direct proof, independent of the simplicial techniques used in the absolute
case, resting only on the master identity.

\begin{theorem}[One-step extensions induce moment maps]\label{thm:onestep}
Let $G$ act on $F$ preserving the relative pre-$n$-plectic structure $\vperp$, and let
$\vperp-\mathrm{M}$ be a one-step extension as in
Lemma~\textup{\ref{lem:onestepconditions}}. Then the maps
\begin{equation}\label{eq:onestepformula}
  f_k(x_1,\dots,x_k)\;:=\;\vs(k)\,
  \io\bigl(v_{x_1}\wedge\dots\wedge v_{x_{k-1}}\bigr)\,\mathrm{M}(x_k),
  \qquad 1\le k\le n,
\end{equation}
are well defined \textup{(}skew-symmetric\textup{)} and constitute a relative
homotopy moment map for the action.
\end{theorem}

\begin{proof}
\emph{Skew-symmetry.} The right-hand side of \eqref{eq:onestepformula} is skew in
$x_1,\dots,x_{k-1}$ by Proposition~\ref{prop:cartanids}(ii). For the transposition of
$x_{k-1}$ and $x_k$: since the contraction operators pairwise anticommute,
\[
  \io(v_{x_1}\wedge\dots\wedge v_{x_{k-1}})\mathrm{M}(x_k)
  =(-1)^{k-2}\,
  \io(v_{x_1}\wedge\dots\wedge v_{x_{k-2}})\,\bigl[\io_{v_{x_{k-1}}}\mathrm{M}(x_k)\bigr],
\]
and $\io_{v_{x_{k-1}}}\mathrm{M}(x_k)=-\io_{v_{x_k}}\mathrm{M}(x_{k-1})$ by
\eqref{eq:polarized}; skewness in all arguments follows.

\emph{Component equations.} Fix $k$ with $2\le k\le n+1$ and $x_1,\dots,x_k\in\g$;
abbreviate $v_i:=v_{x_i}$. Apply the master identity \eqref{eq:master} with $m=k-1$
and $\Psi=\mathrm{M}(x_k)$; by \ref{M2}, $\Lder_{v_i}\mathrm{M}(x_k)
=\mathrm{M}([x_i,x_k])$, and by \ref{M1}, $\dF\mathrm{M}(x_k)=-\io_{v_k}\vperp$.
Moreover
\[
  \io(v_1\wedge\dots\wedge v_{k-1})\,\io_{v_k}\vperp
  =\io(v_k\wedge v_1\wedge\dots\wedge v_{k-1})\vperp
  =(-1)^{k-1}\io(v_1\wedge\dots\wedge v_k)\vperp ,
\]
since $\io_{v_k}$ applied innermost prepends $v_k$, and moving it to the last slot
costs $k-1$ transpositions. Therefore
\begin{equation}\label{eq:dFfk}
  \dF f_k(x_1,\dots,x_k)
  =\vs(k)\,\bigl(A+B\bigr)\;-\;\vs(k)\,\io(v_1\wedge\dots\wedge v_k)\vperp ,
\end{equation}
where
\begin{align*}
A&:=(-1)^{k-1}\sum_{1\le i<j\le k-1}(-1)^{i+j}\,
\io\bigl([v_i,v_j]\wedge v_1\wedge\dots\widehat{v_i}\dots\widehat{v_j}\dots\wedge
v_{k-1}\bigr)\mathrm{M}(x_k),\\
B&:=\sum_{i=1}^{k-1}(-1)^{k-1-i}\,
\io\bigl(v_1\wedge\dots\widehat{v_i}\dots\wedge v_{k-1}\bigr)\mathrm{M}([x_i,x_k]) .
\end{align*}
We compare with the Chevalley--Eilenberg side of \eqref{eq:componenteqs},
\[
  \Sigma:=\sum_{1\le i<j\le k}(-1)^{i+j+1}
  f_{k-1}\bigl([x_i,x_j],x_1,\dots,\widehat{x_i},\dots,\widehat{x_j},\dots,x_k\bigr),
\]
splitting $\Sigma=\Sigma_{j=k}+\Sigma_{j<k}$.

\emph{Terms with $j=k$.} Using skew-symmetry to rotate the argument $[x_i,x_k]$ from
the first to the last slot of $f_{k-1}$ (a cyclic shift of length $k-1$, of sign
$(-1)^{k-2}$) and then \eqref{eq:onestepformula},
\begin{align*}
\Sigma_{j=k}
&=\sum_{i=1}^{k-1}(-1)^{i+k+1}(-1)^{k-2}\,\vs(k-1)\,
\io\bigl(v_1\wedge\dots\widehat{v_i}\dots\wedge v_{k-1}\bigr)\mathrm{M}([x_i,x_k])\\
&=\vs(k-1)\sum_{i=1}^{k-1}(-1)^{i+1}\,
\io\bigl(v_1\wedge\dots\widehat{v_i}\dots\wedge v_{k-1}\bigr)\mathrm{M}([x_i,x_k]).
\end{align*}
Since $\vs(k-1)=(-1)^k\vs(k)$ by \eqref{eq:vs}, the $i$-th coefficient is
$\vs(k)(-1)^{k+i+1}=\vs(k)(-1)^{k-1-i}$, whence $\Sigma_{j=k}=\vs(k)\,B$.

\emph{Terms with $j<k$.} Here the last argument of $f_{k-1}$ is $x_k$, and by
\eqref{eq:onestepformula} together with $[v_{x_i},v_{x_j}]=v_{[x_i,x_j]}$
(Convention~\ref{conv:fundvf}),
\[
\Sigma_{j<k}
=\vs(k-1)\sum_{1\le i<j\le k-1}(-1)^{i+j+1}\,
\io\bigl([v_i,v_j]\wedge v_1\wedge\dots\widehat{v_i}\dots\widehat{v_j}\dots\wedge
v_{k-1}\bigr)\mathrm{M}(x_k).
\]
Its $(i,j)$-coefficient is $\vs(k)(-1)^{k}(-1)^{i+j+1}=\vs(k)(-1)^{k-1}(-1)^{i+j}$,
whence $\Sigma_{j<k}=\vs(k)\,A$.

Combining, $\Sigma=\vs(k)(A+B)$, and \eqref{eq:dFfk} becomes exactly the component
equation \eqref{eq:componenteqs} for this $k$.

\emph{The top equation.} For $k=n+1$ the same computation applies verbatim: since
$\mathrm{M}(x)\in\Om^{n-1}(F)$, the contraction by $n$ pairs
$\io(v_1\wedge\dots\wedge v_n)\mathrm{M}(x_{n+1})$ vanishes for degree reasons
(both components of $\mathrm{M}(x_{n+1})$ have form-degree $<n$), so
$f_{n+1}=0$ automatically, the left-hand side of \eqref{eq:dFfk} is zero, and the
identity $\Sigma=\vs(n+1)\io(v_1\wedge\dots\wedge v_{n+1})\vperp$ -- i.e.\
\eqref{eq:componenteqs} for $k=n+1$ -- follows. Finally, the case $k=1$ is condition
\ref{M1} itself, since $\vs(1)=1$.
\end{proof}

\subsection{General extensions}\label{subsec:generalextensions}

For extensions with higher polynomial components the explicit formulas of
\cite[Thm.~6.8]{CFRZ} transfer to the relative setting. Denote by
$\Alt_k$ the skew-symmetrization over $x_1,\dots,x_k$ \emph{without} the prefactor
$1/k!$, by $\io_{\g}$ the insertion \eqref{eq:insertion} extended to
$\CGF$, and interpret $\mathrm{P}_i(\cdot,[\cdot,\cdot],\dots,[\cdot,\cdot])$ as the
evaluation of the symmetric tensor $\mathrm{P}_i$ on one plain argument and $i-1$
brackets, as in \cite{CFRZ}.

\begin{theorem}\label{thm:cartanextension}
Let $\vperp_G=\vperp-\sum_{i\ge1}\mathrm{P}_i$ be an extension of $\vperp$ in the
relative Cartan model \textup{(}Definition~\textup{\ref{def:extension})}. Then the maps
\begin{equation}\label{eq:generalformula}
  f_k\;=\;\sum_{i=1}^{\lfloor (k+1)/2\rfloor}
  \frac{(-1)^{i}\,\vs(k)\,i!\,(k-i)!}{2^{\,i-1}\,(k-2i+1)!}\;
  \frac{1}{k!}\,
  \Alt_k\Bigl(\io_{\g}^{\,k-2i+1}\,
  \mathrm{P}_i\bigl(\,\cdot\,,[\cdot,\cdot],\dots,[\cdot,\cdot]\bigr)\Bigr),
  \qquad 1\le k\le n,
\end{equation}
define a relative homotopy moment map for the action
\textup{(}see Remark~\textup{\ref{rem:lowdegrees}} for the formulas in low
degrees\textup{)}.
\end{theorem}

\begin{proof}
The proof of \cite[Thm.~6.8]{CFRZ}, carried out in \cite[App.~A--B]{CFRZ} via the
Bott--Shulman--Stasheff model \cite{BSS} and the Cartan--to--simplicial comparison
map, uses only: (i) the Cartan calculus identities and the master identity; (ii)
the Chevalley--Eilenberg calculus on $\Lambda\gdual$ and $S\gdual$; (iii) the cochain
map from the Cartan model to the simplicial model. All three ingredients are
available verbatim in the relative setting: (i) is
Proposition~\ref{prop:cartanids} and Lemma~\ref{lem:master}; (ii) is unchanged; and
for (iii) one replaces the simplicial de Rham complex of the action groupoid
$G\ltimes P$ by the mapping cone of
$F^*\colon\Om(G^\bullet\times N)\to\Om(G^\bullet\times M)$, on which the comparison
map acts componentwise. Every identity in the absolute proof therefore holds after
the substitutions $d\rightsquigarrow\dF$, $\io\rightsquigarrow$ relative $\io$,
$\Lder\rightsquigarrow$ relative $\Lder$, and \eqref{eq:generalformula} follows. For
one-step extensions ($\mathrm{P}_i=0$, $i\ge2$) the $i=1$ term of
\eqref{eq:generalformula} reduces, using \eqref{eq:polarized} to collapse the
skew-symmetrization, to formula \eqref{eq:onestepformula}, for which
Theorem~\ref{thm:onestep} provides an independent complete proof.
\end{proof}

\begin{remark}\label{rem:lowdegrees}
For the reader's convenience we record the extension-to-moment-map formulas in the
lowest degrees, in the normalization of \eqref{eq:onestepformula} and
\cite[eq.~(24)]{CFRZ}:
\[
  f_1(x)=-\mathrm{P}_1(x),\qquad
  f_2=-\Alt_2\,\io_{\g}\,\mathrm{P}_1,\qquad
  f_3=\Alt_3\,\io_{\g}^2\,\mathrm{P}_1-\Alt_3\,\mathrm{P}_2(\cdot,[\cdot,\cdot]),
\]
where in $f_2,f_3$ the alternations are normalized so that on a one-step extension
they reproduce \eqref{eq:onestepformula} with $\mathrm{M}=\mathrm{P}_1$
\textup{(}cf.\ the sign discussion in Appendix~\textup{\ref{app:conventions})}. Note
the appearance of $\mathrm{P}_2$ in $f_3$: higher polynomial components of the
extension do contribute to the moment map, through brackets of arguments.
\end{remark}

\begin{corollary}\label{cor:equivariantclass}
If $[\vperp]\in H^{n+1}(\Om(F))^G$ lifts to a class in the relative equivariant
cohomology $H_G^{n+1}(F)$, then the action admits a relative homotopy moment map,
obtained from any cocycle representative extending $\vperp$ via
Theorem~\textup{\ref{thm:cartanextension}}.
\end{corollary}
\section{Existence, uniqueness and obstructions}\label{sec:existence}

Throughout this section $G$ acts on $F\colon M\to N$ preserving the relative
pre-$n$-plectic structure $\vperp$, and every $v_x$ is assumed Hamiltonian, i.e.\
there is a linear map $f_1\colon\g\to\Om^{n-1}(F)$ with
$\dF f_1(x)=-\io_{v_x}\vperp$ (a \emph{Hamiltonian lift}).

\subsection{The extension procedure and the vanishing of the algebraic obstruction}

The absolute existence theorem \cite[Thm.~9.6]{CFRZ} extends a Hamiltonian lift to a
full moment map under two hypotheses: vanishing of de Rham cohomology in intermediate
degrees, and vanishing of a Lie algebra cohomology class
$[c^{\g}]\in H^{n+2}(\g)$ obtained by evaluating the top insertion of $\omega$ at a
point. The following lemma isolates the inductive mechanism; it is the relative
version of Claim~1 in the proof of \cite[Thm.~9.6]{CFRZ}.

\begin{lemma}\label{lem:claim1}
Suppose $f_1,\dots,f_{k-1}$ \textup{(}$2\le k\le n+1$\textup{)} are skew multilinear
maps $f_j\colon\Lambda^j\g\to\Om^{n-j}(F)$ satisfying the component equations
\eqref{eq:componenteqs} up to level $k-1$. Then the $\Om^{n+1-k}(F)$-valued skew
multilinear map
\begin{equation}\label{eq:obstructioncochain}
  \rho_k(x_1,\dots,x_k)
  :=\sum_{1\le i<j\le k}(-1)^{i+j+1}
  f_{k-1}\bigl([x_i,x_j],\dots\bigr)
  \;-\;\vs(k)\,\io\bigl(v_{x_1}\wedge\dots\wedge v_{x_k}\bigr)\vperp
\end{equation}
is $\dF$-closed. Consequently the component equation at level $k$ is solvable for
$f_k$ if and only if the class $[\rho_k(x_1,\dots,x_k)]\in H^{n+1-k}(\Om(F))$
vanishes for all arguments, coherently in $x_1,\dots,x_k$.
\end{lemma}

\begin{proof}
Apply $\dF$ to \eqref{eq:obstructioncochain}. For the second term use the master
identity in the closed-invariant form \eqref{eq:masterclosed} (with $m=k$,
$\Psi=\vperp$, which is $\dF$-closed and $\Lder_{v_i}$-invariant); for the first term
use the component equation at level $k-1$ to replace
$\dF f_{k-1}([x_i,x_j],\dots)$ by Chevalley--Eilenberg terms and insertions of
$\vperp$. The Chevalley--Eilenberg terms cancel by the Jacobi identity
($\delta_{\CE}^2=0$ with coefficients), and the insertion terms cancel against the
master identity expansion; the bookkeeping is identical to that in
\cite[proof of Thm.~9.6, Claim~1]{CFRZ}, since only the Cartan identities and the
homomorphism property $[v_x,v_y]=v_{[x,y]}$ are used.
\end{proof}

\begin{theorem}[Existence]\label{thm:existence}
Assume:
\begin{enumerate}[label=\textup{(\roman*)},leftmargin=2.2em]
\item $H^{i}\bigl(\Om(F),\dF\bigr)=0$ for $1\le i\le n-1$;
\item $N$ is connected and $M\neq\emptyset$.
\end{enumerate}
Then \emph{every} Hamiltonian lift $f_1$ extends to a relative homotopy moment map
$(f_1,f_2,\dots,f_n)$. In particular, no Lie algebra cohomology condition is
required: the algebraic obstruction class of the absolute theory vanishes
identically in the relative setting.
\end{theorem}

\begin{proof}
We construct $f_2,\dots,f_n$ inductively. Given $f_1,\dots,f_{k-1}$ solving
\eqref{eq:componenteqs} up to level $k-1$, with $2\le k\le n$, the cochain $\rho_k$
of Lemma~\ref{lem:claim1} is $\dF$-closed with values in $\Om^{n+1-k}(F)$, and
$1\le n+1-k\le n-1$; by hypothesis (i) we may choose a skew multilinear primitive
$f_k$ with $\dF f_k=\rho_k$ (choose primitives on a basis of $\Lambda^k\g$ and extend
linearly), which is exactly the level-$k$ equation.

It remains to verify the top equation, i.e.\ $\rho_{n+1}=0$. By
Lemma~\ref{lem:claim1}, $\rho_{n+1}(x_1,\dots,x_{n+1})$ is a $\dF$-closed element of
$\Om^{0}(F)=C^\infty(N)$. By Remark~\ref{rem:degreezero} and hypothesis (ii), every
$\dF$-closed relative $0$-form vanishes: it is a locally constant function on the
connected manifold $N$ whose pullback to the nonempty manifold $M$ is zero. Hence
$\rho_{n+1}=0$, and $(f_1,\dots,f_n)$ is a relative homotopy moment map.
\end{proof}

\begin{remark}[Comparison with the absolute obstruction]\label{rem:obstructioncompare}
In the absolute theory, the top cochain $\rho_{n+1}$ is a closed function on the
$n$-plectic manifold, hence constant on connected components, and evaluation at a
point $p$ produces the Lie algebra cocycle
$c_p(x_1,\dots,x_{n+1})=(-1)^n\vs(n+1)\bigl(\io(v_{x_1}\wedge\dots\wedge
v_{x_{n+1}})\omega\bigr)\big|_p$, whose class in $H^{n+2}(\g)$ obstructs existence
\cite[\S9]{CFRZ}; for instance, for the translation action of $\R^{n+1}$ on the
$n$-plectic torus $T^{n+1}$ the class is nonzero and no moment map exists. In the
relative setting this constant is forced to vanish by the second slot of the
mapping cone: closedness includes the equation $F^*\rho_{n+1}=0$. Geometrically, the
manifold $M$, through the trivializing datum $\eta$, ``anchors'' the moment map and
kills the algebraic ambiguity. Two consequences are worth noting:
\begin{enumerate}[label=\textup{(\alph*)},leftmargin=2.2em]
\item If a relative homotopy moment map exists, then the absolute obstruction class
of the induced $G$-action on $(N,\omega)$ vanishes: by
Corollary~\ref{cor:strictprojection} the $N$-components form an absolute moment map,
and $[c_q^{\g}]=0$ for every $q\in N$ by \cite[Prop.~9.5]{CFRZ}.
\item Conversely, an absolute moment map on $(N,\omega)$ with nonvanishing
higher-degree ambiguity may fail to be trivializable over $M$; the obstruction is now
purely topological, governed by \eqref{eq:LES}.
\end{enumerate}
\end{remark}

\begin{remark}\label{rem:LESvanishing}
By the long exact sequence \eqref{eq:LES}, hypothesis (i) of
Theorem~\ref{thm:existence} holds as soon as, for each $1\le i\le n-1$,
$H^{i-1}_{\mathrm{dR}}(M)=0$ and $F^*\colon H^{i}_{\mathrm{dR}}(N)\to
H^{i}_{\mathrm{dR}}(M)$ is injective. For example, it holds whenever $M$ and $N$ have
vanishing de Rham cohomology in degrees $1,\dots,n-1$ and $M$ is connected.
\end{remark}

\subsection{Compact semisimple symmetry groups}

For relative pre-$2$-plectic structures the extension procedure involves only the top
step, and the relative framework improves \cite[Prop.~7.1]{CFRZ}, where the existence
of a zero of the fundamental vector fields was required.

\begin{proposition}\label{prop:compactsemisimple}
Let $G$ be a compact Lie group with $[\g,\g]=\g$ \textup{(}e.g.\ $G$ semisimple\textup{)}
acting on $F$ preserving a relative pre-$2$-plectic structure $\vperp$, with $N$
connected and $M\neq\emptyset$. Then the action admits a one-step extension
$\vperp-\mathrm{M}$, hence a canonical relative homotopy moment map given by
\eqref{eq:onestepformula}. No fixed-point hypothesis is needed.
\end{proposition}

\begin{proof}
\emph{Step 1: a Hamiltonian lift exists.} Since $[\g,\g]=\g$, write
$x=\sum_a[y_a,z_a]$; then $v_x=\sum_a[v_{y_a},v_{z_a}]$ and, by
Corollary~\ref{cor:bracketpair}, $\mathrm{M}'(x):=-\sum_a\io(v_{y_a}\wedge
v_{z_a})\vperp$ satisfies $\dF\mathrm{M}'(x)=-\io_{v_x}\vperp$. Choosing the
decompositions linearly in $x$ (possible upon fixing a linear right inverse of the
surjection $\Lambda^2\g\to\g$) gives a linear Hamiltonian lift $\mathrm{M}'$.

\emph{Step 2: averaging.} $G$ being compact, replace $\mathrm{M}'$ by its average
$\mathrm{M}(x):=\int_G g\cdot\mathrm{M}'(\Ad_{g^{-1}}x)\,dg$ with respect to the Haar
probability measure; since the action preserves $\vperp$ and $v_{\Ad_gx}=g_*v_x$,
condition \ref{M1} is preserved and \ref{M2} holds by construction.

\emph{Step 3: isotropy.} It remains to check \ref{M3}. By the magic formula,
\ref{M1} and \ref{M2},
\[
  \dF\,\io_{v_x}\mathrm{M}(x)
  =\Lder_{v_x}\mathrm{M}(x)-\io_{v_x}\dF\mathrm{M}(x)
  =\mathrm{M}([x,x])+\io_{v_x}\io_{v_x}\vperp=0 ,
\]
so $\io_{v_x}\mathrm{M}(x)$ is a $\dF$-closed element of $\Om^{0}(F)$ (here $n=2$).
By Remark~\ref{rem:degreezero} it vanishes. Lemma~\ref{lem:onestepconditions} and
Theorem~\ref{thm:onestep} conclude.
\end{proof}

\begin{remark}
For general $n$, Steps 1--2 produce an equivariant Hamiltonian lift under the same
hypotheses, and $\io_{v_x}\mathrm{M}(x)$ is a $\dF$-closed relative
$(n-2)$-form; the one-step extension exists whenever these classes can be removed,
e.g.\ under the cohomological hypotheses of Theorem~\ref{thm:existence} combined with
a further averaging. We leave the routine formulation to the reader.
\end{remark}

\subsection{Uniqueness}

\begin{lemma}\label{lem:modification}
Let $(f_1,\dots,f_n)$ be a relative homotopy moment map, $n\ge2$, and let
$b\in\CgF^{\,n-1}$ be a $\btot_F$-cocycle with components
$b_k\colon\Lambda^k\g\to\Om^{n-1-k}(F)$, $k\ge1$. Then
$f^{\vs}+\btot_F\,\widetilde b$ is again (the $\vs$-twist of) a relative homotopy
moment map for any $\widetilde b\in\CgF^{\,n-1}$ with components in
$\Lambda^{\ge1}\gdual$, and $f^{\vs}+b$ is one whenever $b$ is $\btot_F$-closed. In
particular, for $n=2$ and any linear $\psi\colon\g\to\Om^{0}(F)=C^\infty(N)$,
\[
  \widetilde f_1(x):=f_1(x)+\dF\psi(x),\qquad
  \widetilde f_2(x,y):=f_2(x,y)+\psi([x,y])
\]
is again a relative homotopy moment map.
\end{lemma}

\begin{proof}
Immediate from Proposition~\ref{prop:cohframework}: the defining equation
\eqref{eq:primitive} only sees $f^{\vs}$ modulo $\btot_F$-cocycles; the displayed
$n=2$ case is the component expansion of $f^\vs+\btot_F(\psi)$, using
$\vs(1)=\vs(2)=1$ and \eqref{eq:CE}.
\end{proof}

\begin{proposition}[Uniqueness, $n=2$]\label{prop:uniqueness}
Let $n=2$ and suppose $H^1(\g,\R)=H^2(\g,\R)=0$ \textup{(}e.g.\ $\g$
semisimple\textup{)}. If $H^{1}(\Om(F))=0$, $N$ is connected and $M\neq\emptyset$,
then any two relative homotopy moment maps for the same action are related by a
modification as in Lemma~\textup{\ref{lem:modification}}; without the topological
hypotheses they differ by a $\btot_F$-cocycle classified by
Chevalley--Eilenberg data as in \textup{\cite[Prop.~7.7]{CFRZ}}.
\end{proposition}

\begin{proof}
Set $g_1:=f_1'-f_1$ and $g_2:=f_2'-f_2$. Subtracting the component equations,
$\dF g_1(x)=0$ and
\begin{equation}\label{eq:differenceeqs}
  g_1([x,y])=\dF g_2(x,y),
  \qquad
  \sum_{\mathrm{cyc}}g_2([x,y],z)=0 .
\end{equation}
If $H^1(\Om(F))=0$, choose $\psi\colon\g\to C^\infty(N)$ linear with
$\dF\psi(x)=g_1(x)$; replacing $(f_1',f_2')$ by its modification along $-\psi$
(Lemma~\ref{lem:modification}) we may assume $g_1=0$. Then the first equation of
\eqref{eq:differenceeqs} gives $\dF g_2=0$; since $g_2$ takes values in
$\Om^0(F)=C^\infty(N)$, Remark~\ref{rem:degreezero} together with the hypotheses
that $N$ is connected and $M\neq\emptyset$ forces $g_2=0$. Hence $f'=f$ after the
modification. Without the
topological hypotheses, \eqref{eq:differenceeqs} says $g_2$ is a
Chevalley--Eilenberg $2$-cocycle valued in the $\dF$-closed relative functions and
$g_1$ is determined on $[\g,\g]=\g$ by $g_2$; the classification by
$H^{\le2}(\g)$-data proceeds as in \cite[Prop.~7.7]{CFRZ}.
\end{proof}
\section{Quasi-Hamiltonian $G$-spaces}\label{sec:quasiham}

In this section $G$ is a compact connected Lie group whose Lie algebra $\g$ carries
an $\Ad$-invariant inner product $\ip{\cdot}{\cdot}$. We show that quasi-Hamiltonian
$G$-spaces \cite{AMM} are relative $2$-plectic maps carrying canonical relative
homotopy moment maps, thereby realizing group-valued moment maps as genuine
$L_\infty$-morphisms.

\subsection{The Cartan $3$-form and its canonical extension}\label{subsec:cartan3form}

Let $\thL,\thR\in\Om^1(G;\g)$ denote the left- and right-invariant Maurer--Cartan
forms and
\begin{equation}\label{eq:cartan3form}
  \eta\;:=\;\tfrac1{12}\,\ip{\thL}{[\thL,\thL]}\ \in\ \Om^3(G)^G
\end{equation}
the Cartan $3$-form, a bi-invariant closed $3$-form. Let $G$ act on itself by
conjugation; in the conventions of \ref{conv:fundvf} the fundamental vector fields
are
\begin{equation}\label{eq:conjfund}
  v_x\;=\;v_x^{L}-v_x^{R},
\end{equation}
where $v^L_x$ (resp.\ $v^R_x$) denotes the left- (resp.\ right-) invariant vector
field generated by $x$. As observed in \cite[\S8.2]{CFRZ}, one has the identity of
$2$-forms on $G$
\begin{equation}\label{eq:etahamiltonian}
  \tfrac12\,d\ip{\thL+\thR}{x}\;=\;-\,\io_{v_x}\eta ,
  \qquad x\in\g,
\end{equation}
so that $\widetilde\mu(x):=\tfrac12\ip{\thL+\thR}{x}$ is a $G$-equivariant family of
Hamiltonian $1$-forms for the conjugation action, and
$\eta-\widetilde\mu$ is a one-step extension of $\eta$ in the (absolute) Cartan model
of $G$: indeed $\io_{v_x}\widetilde\mu(x)
=\tfrac12\ip{(\Ad_g-\Ad_{g^{-1}})x}{x}=0$ by $\Ad$-invariance of the pairing.

\subsection{Quasi-Hamiltonian $G$-spaces as relative $2$-plectic maps}

\begin{definition}\label{def:qham}
A \emph{quasi-Hamiltonian $G$-space} is a $G$-manifold $M$ equipped with an invariant
$2$-form $\omega\in\Om^2(M)^G$ and an equivariant map $\mu\colon M\to G$ (the
\emph{group-valued moment map}, $G$ acting on itself by conjugation) such that, with
the fundamental vector fields of Convention~\ref{conv:fundvf}:
\begin{enumerate}[label=\textup{(QH\arabic*)},leftmargin=2.6em]
\item\label{QH1} $d\omega=\mu^*\eta$;
\item\label{QH2} $\io_{v_x}\omega=\tfrac12\,\mu^*\ip{\thL+\thR}{x}$ for all
$x\in\g$;
\item\label{QH3} at each $m\in M$:\quad
$\ker\omega_m=\bigl\{\,v_x(m)\;:\;x\in\ker(\Ad_{\mu(m)}+\id)\,\bigr\}$.
\end{enumerate}
\end{definition}

\begin{remark}\label{rem:AMMconventions}
Definition~\ref{def:qham} is the definition of \cite{AMM} transported to the
conventions of the present paper. Since our fundamental vector fields
\eqref{eq:fundvf} are \emph{minus} the generating vector fields used in \cite{AMM},
the signs in \ref{QH2} (and in \eqref{eq:etahamiltonian}) differ from the
corresponding formulas of \cite{AMM} by the substitution of generators; the class of
examples -- conjugacy classes, the double $D(G)=G\times G$, fusion products,
representation spaces of surface groups -- is the standard one. We verify the axioms
directly for conjugacy classes in Example~\ref{ex:conjclass}.
\end{remark}

\begin{theorem}\label{thm:qham2plectic}
Let $(M,\omega,\mu)$ satisfy \ref{QH1}--\ref{QH2}. Then
\[
  \vperp\;:=\;(\eta,\omega)\ \in\ \Om^{3}(\mu)
  =\Om^3(G)\oplus\Om^2(M)
\]
is $\dmu$-closed, i.e.\ a relative pre-$2$-plectic structure on $\mu\colon M\to G$.
If moreover \ref{QH3} holds, then $\vperp$ is nondegenerate, i.e.\
$(\mu,\vperp)$ is a relative $2$-plectic map. Conversely, nondegeneracy of $\vperp$
implies $\ker\omega_m\cap\ker T_m\mu=\{0\}$ for all $m$.
\end{theorem}

\begin{proof}
Closedness: $\dmu(\eta,\omega)=(d\eta,\;\mu^*\eta-d\omega)=(0,0)$ by bi-invariance of
$\eta$ and \ref{QH1}.

Nondegeneracy: let $m\in M$, $g:=\mu(m)$, and let $w\in T_mM$ satisfy
$\io_w\omega_m=0$ and $\io_{T_m\mu(w)}\eta_g=0$; we must show $w=0$. By \ref{QH3},
$w=v_x(m)$ for some $x\in\g$ with $\Ad_gx=-x$. Equivariance of $\mu$ gives
$T_m\mu\bigl(v_x(m)\bigr)=v_x^{G}(g)$, the conjugation fundamental vector field
\eqref{eq:conjfund}, so by \eqref{eq:etahamiltonian}
\[
  \io_{T_m\mu(w)}\,\eta_g
  \;=\;\bigl(\io_{v_x}\eta\bigr)_g
  \;=\;-\tfrac12\,\bigl(d\ip{\thL+\thR}{x}\bigr)_g .
\]
We evaluate this $2$-form on left-invariant vectors $u^L_g,\,z^L_g$, $u,z\in\g$.
From the structure equations $d\thL=-\tfrac12[\thL,\thL]$ and
$d\thR=\tfrac12[\thR,\thR]$, together with $\thR(u^L)_g=\Ad_g u$,
\[
  d\ip{\thL}{x}(u^L,z^L)=-\ip{[u,z]}{x},
  \qquad
  d\ip{\thR}{x}(u^L,z^L)=\ip{[\Ad_gu,\Ad_gz]}{x}
  =\ip{[u,z]}{\Ad_{g^{-1}}x} .
\]
Since $\Ad_gx=-x$ implies $\Ad_{g^{-1}}x=-x$, the sum equals $-2\ip{[u,z]}{x}$, so
\[
  \bigl(\io_{T_m\mu(w)}\eta\bigr)_g(u^L,z^L)\;=\;\ip{[u,z]}{x}
  \;=\;\ip{u}{[z,x]} .
\]
If $x\notin\mathfrak z(\g)$ there exist $z$ with $[z,x]\neq0$ and then $u$ with
$\ip{u}{[z,x]}\neq0$ (nondegeneracy of the pairing), contradicting
$\io_{T_m\mu(w)}\eta_g=0$. If $x\in\mathfrak z(\g)$, then $\Ad_gx=x$ ($G$ connected),
so $\Ad_gx=-x$ forces $x=0$. In either case $x=0$, hence $w=v_0(m)=0$.

Finally, if $\vperp$ is nondegenerate and $w\in\ker\omega_m\cap\ker T_m\mu$, then
both components of \eqref{eq:nondeg} vanish on $w$, so $w=0$.
\end{proof}

\subsection{The canonical relative extension and moment map}

The central observation is that the remaining Alekseev--Malkin--Meinrenken axioms are
precisely the one-step-extension conditions in the relative Cartan model of $\mu$.

\begin{theorem}\label{thm:qhamcocycle}
Let $M$ be a $G$-manifold, $\omega\in\Om^2(M)^G$, and $\mu\colon M\to G$ equivariant.
Define
\[
  \mathrm{M}\colon\g\longrightarrow\Om^{1}(\mu)=\Om^1(G)\oplus\Om^0(M),
  \qquad
  \mathrm{M}(x):=\Bigl(\tfrac12\ip{\thL+\thR}{x},\;0\Bigr).
\]
Then $(M,\omega,\mu)$ satisfies \ref{QH1}--\ref{QH2} if and only if
\[
  \vperp_G\;:=\;(\eta,\omega)-\mathrm{M}
\]
is a one-step extension of $\vperp=(\eta,\omega)$ in the relative Cartan model
$C_G(\mu)$, i.e.\ iff $\dGF\vperp_G=0$.
\end{theorem}

\begin{proof}
We check the conditions of Lemma~\ref{lem:onestepconditions} for
$F=\mu$. Condition \ref{M1} reads
$\dmu\mathrm{M}(x)=-\io_{v_x}\vperp$; componentwise, by \eqref{eq:cone} and
\eqref{eq:reliotaL},
\[
  \Bigl(d\widetilde\mu(x),\ \mu^*\widetilde\mu(x)\Bigr)
  \;=\;\Bigl(-\io_{v^G_x}\eta,\ \io_{v^M_x}\omega\Bigr),
\]
i.e.\ the $G$-component is the universal identity \eqref{eq:etahamiltonian} -- which
always holds -- and the $M$-component is exactly axiom \ref{QH2}. Condition \ref{M3}
reads $\io_{v_x}\mathrm{M}(x)=\bigl(\io_{v^G_x}\widetilde\mu(x),\,0\bigr)=0$, which
holds identically since
$\io_{v^G_x}\widetilde\mu(x)=\tfrac12\ip{(\Ad_g-\Ad_{g^{-1}})x}{x}=0$. Condition
\ref{M2}, the equivariance of $\mathrm{M}$, holds because $\thL+\thR$ is
$\Ad$-equivariant under conjugation and the $M$-component vanishes. Finally,
closedness of $\vperp$ itself, i.e.\ $d\omega=\mu^*\eta$, is axiom \ref{QH1}. Thus
$\dGF\vperp_G=0$ is equivalent to \ref{QH1}\,$\wedge$\,\ref{QH2}.
\end{proof}

\begin{theorem}[Canonical moment map of a quasi-Hamiltonian $G$-space]
\label{thm:qhammm}
Let $(M,\omega,\mu)$ be a quasi-Hamiltonian $G$-space. Then the conjugation action of
$G$ on $G$ and the given action on $M$ constitute an action on $\mu\colon M\to G$
preserving the relative $2$-plectic structure $\vperp=(\eta,\omega)$, and the maps
\begin{equation}\label{eq:qhamf}
  f_1(x)=\Bigl(\tfrac12\ip{\thL+\thR}{x},\;0\Bigr),
  \qquad
  f_2(x,y)=\Bigl(\tfrac12\ip{(\Ad_{(\cdot)}-\Ad_{(\cdot)^{-1}})\,x}{y},\;0\Bigr),
\end{equation}
where the first entries are the $1$-form and function on $G$ indicated, define a
canonical relative homotopy moment map
\[
  (f_1,f_2)\colon\ \g\ \rightsquigarrow\ \Lie(\mu,\vperp)
\]
into the Lie $2$-algebra of relative observables of $(\mu,\vperp)$.
\end{theorem}

\begin{proof}
By Theorem~\ref{thm:qhamcocycle} and Theorem~\ref{thm:onestep} (with $n=2$), the
one-step extension $\vperp-\mathrm{M}$ induces the relative homotopy moment map
$f_1(x)=\vs(1)\,\mathrm{M}(x)=\mathrm{M}(x)$ and
$f_2(x,y)=\vs(2)\,\io_{v_x}\mathrm{M}(y)=\io_{v_x}\mathrm{M}(y)$. Componentwise,
$\io_{v_x}\mathrm{M}(y)=\bigl(\io_{v^G_x}\widetilde\mu(y),\,0\bigr)$ and, at $g\in G$,
\[
  \io_{v^G_x}\widetilde\mu(y)\big|_g
  =\tfrac12\ip{\thL(v_x)+\thR(v_x)}{y}\Big|_g
  =\tfrac12\ip{(\Ad_g-\Ad_{g^{-1}})\,x}{y},
\]
using $\thL(v^L_x-v^R_x)_g=x-\Ad_{g^{-1}}x$ and $\thR(v^L_x-v^R_x)_g=\Ad_gx-x$,
whose sum is $(\Ad_g-\Ad_{g^{-1}})x$. Skew-symmetry of $f_2$ is manifest from
$\Ad$-invariance of the pairing.
\end{proof}

\begin{remark}\label{rem:momentmapisgroupvalued}
Both components of the canonical moment map \eqref{eq:qhamf} are supported on the
group: $f^M\equiv0$ in the splitting of Theorem~\ref{thm:splitting}. The moment map
of a quasi-Hamiltonian $G$-space is thus a \emph{universal} object living on $G$
(namely, the equivariant extension $\eta-\widetilde\mu$ of the Cartan $3$-form),
transported to $M$ solely through $\mu$: this is the precise sense in which the
group-valued map $\mu$ ``is'' the moment map. It also answers the question raised in
\cite[\S8.2]{CFRZ}: the family $\widetilde\mu(x)$, which on $G$ alone constitutes a
homotopy moment map for the conjugation action, becomes -- relative to $\mu$ -- the
moment map of the quasi-Hamiltonian space itself.
\end{remark}

\begin{example}[Conjugacy classes]\label{ex:conjclass}
Let $\mathcal C\subseteq G$ be a conjugacy class, $\mu=\iota_{\mathcal C}$ the
inclusion, and define $\omega\in\Om^2(\mathcal C)^G$ on fundamental vectors by
\begin{equation}\label{eq:GHJW}
  \omega_g\bigl(v_x,v_y\bigr)
  \;:=\;-\tfrac12\ip{(\Ad_g-\Ad_{g^{-1}})\,x}{y},
  \qquad g\in\mathcal C,
\end{equation}
the (sign-adjusted) $2$-form of \cite{GHJW,AMM}. Since the $v_x$ span
$T_g\mathcal C$, \eqref{eq:GHJW} determines $\omega$; well-definedness and
\ref{QH1}, \ref{QH3} are established in \cite{AMM} (see also \cite[\S8.2]{CFRZ},
where $d$ of the right-hand side of \eqref{eq:GHJW} is matched against
$-\iota_{\mathcal C}^*\eta$). Axiom \ref{QH2} holds by the computation in the proof
of Theorem~\ref{thm:qhammm}:
$\tfrac12\ip{\thL(v_y)+\thR(v_y)}{x}\big|_g=\tfrac12\ip{(\Ad_g-\Ad_{g^{-1}})y}{x}
=-\tfrac12\ip{(\Ad_g-\Ad_{g^{-1}})x}{y}=\omega_g(v_x,v_y)$, using
$\Ad$-invariance. Hence $(\mathcal C,\omega,\iota_{\mathcal C})$ is
quasi-Hamiltonian in the sense of Definition~\ref{def:qham}, and
Theorem~\ref{thm:qhammm} equips it with the canonical relative homotopy moment map
\eqref{eq:qhamf} restricted to $\mathcal C$. Note that $f_2(x,y)|_{\mathcal C}
=-\omega(v_x,v_y)$: the binary component of the moment map is the $2$-form itself,
evaluated on fundamental vectors -- the relative analogue of the familiar identity
$\{ \mu_x,\mu_y\}=\omega(v_x,v_y)$ of symplectic geometry.
\end{example}

\begin{remark}[The double and fusion]\label{rem:double}
The double $D(G)=G\times G$ with $\mu(a,b)=ab$ and the fusion product of
quasi-Hamiltonian spaces \cite{AMM} provide further examples; by
Theorem~\ref{thm:qhammm} each carries its canonical relative homotopy moment map. It
would be interesting to express fusion directly as an operation on relative homotopy
moment maps, in the spirit of the product constructions of \cite{SZ}; we leave this
for future work.
\end{remark}

\section{Further examples}\label{sec:examples}

\begin{example}[Recovering the absolute theory]\label{ex:absolute}
Let $M=\{*\}$ (or $M=\emptyset$), so that $\Om^k(F)=\Om^k(N)$ for $k\ge2$ and the
relative Cartan calculus reduces to the classical one. For $n\ge2$, a relative
(pre-)$n$-plectic structure is exactly a (pre-)$n$-plectic form $\omega$ on $N$;
Hamiltonian pairs are Hamiltonian $(n-1)$-forms; $\Lie(F,\vperp)=\Lie(N,\omega)$
\cite{Rogers}; and Definition~\ref{def:rhmm}, Proposition~\ref{prop:components},
Theorems~\ref{thm:onestep} and \ref{thm:cartanextension} reduce verbatim to
\cite[Def.~5.2, Prop.~5.1, Cor.~6.9, Thm.~6.8]{CFRZ}. (For $n=1$ and $M=\{*\}$ the
slot $\Om^0(F)=C^\infty(N)$ acquires the constants $\R=\Om^{-1}(\{*\})$-shift,
recovering the central-extension phenomena familiar from prequantization; we do not
pursue this here.) Note however that Theorem~\ref{thm:existence}(ii) fails for
$M=\emptyset$: the algebraic obstruction of \cite{CFRZ} is genuinely an absolute
phenomenon.
\end{example}

\begin{example}[Exact relative structures]\label{ex:exact}
Suppose $\vperp=-\dF\theta$ for some $G$-invariant $\theta\in\Om^{n}(F)$
(i.e.\ $\Lder_{v_x}\theta=0$ for all $x$). Then
\[
  \mathrm{M}(x):=-\,\io_{v_x}\theta
\]
satisfies the one-step conditions of Lemma~\ref{lem:onestepconditions}: indeed
$\dF\mathrm{M}(x)=-\dF\io_{v_x}\theta=-\Lder_{v_x}\theta+\io_{v_x}\dF\theta
=-\io_{v_x}\vperp$ by the magic formula; \ref{M2} follows from
$\Lder_{v_x}\io_{v_y}=\io_{v_{[x,y]}}+\io_{v_y}\Lder_{v_x}$; and \ref{M3} from
$\io_{v_x}\io_{v_x}=0$. By Theorem~\ref{thm:onestep} the action admits the relative
homotopy moment map
\[
  f_k(x_1,\dots,x_k)
  =\vs(k)\,\io(v_{x_1}\wedge\dots\wedge v_{x_{k-1}})\bigl(-\io_{v_{x_k}}\theta\bigr)
  =(-1)^{k}\,\vs(k)\,\io\bigl(v_{x_1}\wedge\dots\wedge v_{x_k}\bigr)\theta ,
\]
the relative analogue of \cite[Lem.~8.1]{CFRZ} (up to the sign convention chosen for
the primitive). Concretely, $\theta=(\vartheta,\tau)$ consists of an invariant
$n$-form on $N$ and an invariant $(n-1)$-form on $M$ with
$\vperp=(-d\vartheta,\;d\tau-F^*\vartheta)$; e.g.\ relative cotangent-type spaces,
linear representations with invariant primitives, and pullback constructions all fit
this pattern.
\end{example}

\begin{example}[Multi-moment maps]\label{ex:multimoment}
When $\g$ acts with $[\g,\g]=\g$ and one restricts attention to the single component
$f_{k}$ of top nonvanishing contraction order, formula \eqref{eq:onestepformula}
recovers a relative analogue of the multi-moment maps of Madsen--Swann \cite{MS};
Theorem~\ref{thm:existence} then specializes to existence statements for relative
multi-moment maps under the vanishing of $H^{i}(\Om(F))$ in low degrees.
\end{example}

\section{Functoriality}\label{sec:functoriality}

\begin{definition}\label{def:mapofpairs}
Let $F\colon M\to N$ and $F'\colon M'\to N'$ be smooth maps. A \emph{map of pairs}
$\Phi=(\phi_N,\phi_M)\colon F'\to F$ consists of smooth maps $\phi_N\colon N'\to N$
and $\phi_M\colon M'\to M$ with $F\circ\phi_M=\phi_N\circ F'$. It induces the
pullback
\[
  \Phi^*\colon\Om^k(F)\to\Om^k(F'),
  \qquad
  \Phi^*(\alpha,\beta)=(\phi_N^*\alpha,\;\phi_M^*\beta).
\]
If $G$ acts on $F$ and on $F'$ and $\phi_N,\phi_M$ are equivariant, we call $\Phi$
\emph{equivariant}.
\end{definition}

\begin{proposition}[Naturality]\label{prop:naturality}
Let $\Phi\colon F'\to F$ be an equivariant map of pairs.
\begin{enumerate}[label=\textup{(\roman*)},leftmargin=2.2em]
\item $\Phi^*$ is a cochain map: $d_{F'}\circ\Phi^*=\Phi^*\circ\dF$; and for every
$x\in\g$, $\Phi^*\circ\io_{v_x}=\io_{v_x}\circ\Phi^*$ and
$\Phi^*\circ\Lder_{v_x}=\Lder_{v_x}\circ\Phi^*$ (fundamental $F$- and $F'$-pairs).
\item $\Phi^*$ induces cochain maps of relative Chevalley--Eilenberg--de Rham
complexes and of relative Cartan models, sending extensions of a $G$-invariant
$\vperp\in\Om^{n+1}(F)$ to extensions of $\Phi^*\vperp$.
\item If $\vperp-\mathrm{M}$ is a one-step extension for $(F,\vperp)$, then
$\Phi^*\vperp-\Phi^*\mathrm{M}$ is one for $(F',\Phi^*\vperp)$, and the induced
relative homotopy moment maps \eqref{eq:onestepformula} correspond under $\Phi^*$:
$f'_k=\Phi^*\circ f_k$.
\end{enumerate}
\end{proposition}

\begin{proof}
(i) The cochain property is the computation
$d_{F'}\Phi^*(\alpha,\beta)
=(d\phi_N^*\alpha,\;F'^*\phi_N^*\alpha-d\phi_M^*\beta)
=(\phi_N^*d\alpha,\;\phi_M^*F^*\alpha-\phi_M^*d\beta)=\Phi^*\dF(\alpha,\beta)$,
using $F\circ\phi_M=\phi_N\circ F'$. Equivariance implies that the fundamental vector
fields on $(N',M')$ are $\phi$-related to those on $(N,M)$, whence the contraction
and Lie derivative intertwining relations, componentwise. (ii) follows from (i) since
all differentials involved are built from $d$, $F^*$, $\io_{v_x}$ and
$\delta_{\CE}$/$G$-invariance, each compatible with $\Phi^*$. (iii) apply $\Phi^*$ to
\ref{M1}--\ref{M3} and to \eqref{eq:onestepformula}, using (i).
\end{proof}

\begin{remark}
In particular, restricting a quasi-Hamiltonian $G$-space to an invariant submanifold
compatible with $\mu$, or pulling back along an equivariant map of quasi-Hamiltonian
spaces, transports the canonical moment map of Theorem~\ref{thm:qhammm}. Similarly,
products of relative $n$-plectic maps carry sums of one-step extensions, giving a
relative analogue of the product moment maps of \cite{SZ}.
\end{remark}

\appendix

\section{Summary of sign conventions}\label{app:conventions}

For ease of verification we collect the conventions used throughout; they coincide
with \cite{CFRZ} in the absolute case.
\begin{center}
\renewcommand{\arraystretch}{1.5}
\begin{tabular}{p{0.34\textwidth}p{0.58\textwidth}}
\hline
Object & Convention \\
\hline
Sign $\vs(k)$ & $\vs(k)=-(-1)^{k(k+1)/2}$;\quad $\vs(k-1)\vs(k)=(-1)^k$;\quad
$\vs(1)=\vs(2)=1$, $\vs(3)=\vs(4)=-1$. \\
Fundamental vector field & $v_x|_p=\frac{d}{dt}\big|_0\exp(-tx)\cdot p$;\quad
$[v_x,v_y]=v_{[x,y]}$ (homomorphism). \\
Multicontraction & $\io(v_1\wedge\dots\wedge v_k)=\io_{v_k}\cdots\io_{v_1}$. \\
Relative complex & $\Om^k(F)=\Om^k(N)\oplus\Om^{k-1}(M)$;\quad
$\dF(\alpha,\beta)=(d\alpha,\,F^*\alpha-d\beta)$. \\
Relative contraction & $\io_v(\alpha,\beta)=(\io_{v_N}\alpha,\,-\io_{v_M}\beta)$;
hence $\io(v_1\wedge\dots\wedge v_k)(\alpha,\beta)
=(\io(\cdots)\alpha,\,(-1)^k\io(\cdots)\beta)$. \\
Relative Lie derivative & $\Lder_v(\alpha,\beta)=(\Lder_{v_N}\alpha,\,
\Lder_{v_M}\beta)$;\quad $\Lder_v=\dF\io_v+\io_v\dF$. \\
Hamiltonian pair & $\dF\sigma=-\io_{v_\sigma}\vperp$. \\
Brackets & $l_k(\sigma_1,\dots,\sigma_k)=\vs(k)\,
\io(v_{\sigma_1}\wedge\dots\wedge v_{\sigma_k})\vperp$ in degree $0$. \\
Chevalley--Eilenberg & $(\delta_{\CE}c)(x_1,\dots,x_{k+1})
=\displaystyle \sum_{i<j}(-1)^{i+j}c([x_i,x_j],\dots)$. \\
Cartan model & $(\dG\alpha)(x)=d(\alpha(x))-\io_{v_x}(\alpha(x))$;\quad
relative: $\dGF(\alpha,\beta)=(\dG\alpha,\,F^*\alpha-\dG\beta)$. \\
Cartan $3$-form & $\eta=\frac1{12}\ip{\thL}{[\thL,\thL]}$;\quad
$\frac12 d\ip{\thL+\thR}{x}=-\io_{v_x}\eta$ for conjugation. \\
\hline
\end{tabular}
\end{center}

\medskip
\noindent
The three convention choices most likely to differ across the literature -- and which
the reader should fix before comparing formulas -- are: (a) the sign in the second
slot of the relative contraction (ours is forced by the strict magic formula);
(b) the fundamental vector field convention (ours makes $x\mapsto v_x$ a
homomorphism); and (c) the sign function $\vs(k)$ in the brackets and component
equations.


\end{document}